\documentclass[10pt, a4paper, twoside]{amsart}
\usepackage{fancyhdr}
\allowdisplaybreaks

\renewcommand\thesection{\arabic{section}}

\theoremstyle{theorem}
\newtheorem{Thm}{Theorem} [section]
\newtheorem{Prop}[Thm]{Proposition}

\newtheorem{Lem}[Thm]{Lemma}

\newtheorem{Ass}[Thm]{Assumption}

\newtheorem{Def}[Thm]{Definition}

\newtheorem{Rem}[Thm]{Remark}

\newcommand{\PP}{\mathds{P}}

\newcommand{\R}{\mathbb{R}}
\newcommand{\N}{\mathbb{N}}
\newcommand{\borel}{\mathfrak{B}}
\newcommand{\eps}{\varepsilon}
\DeclareMathOperator{\e}{e}

\newcommand{\Lra}{\Longrightarrow}
\newcommand{\Llra}{\Longleftrightarrow}
\newcommand{\ra}{\rightarrow}

\newcommand{\rra}{\rightrightarrows}

\newcommand{\Let}{\coloneqq}

\newcommand{\diff}{\mathrm{d}}

\newcommand{\wt}{\widetilde}

\newcommand{\tr}{^\intercal}

\newcommand{\ball}[2]{\mathbb{B}_{#2}(#1)}		

\newcommand{\Lnorm}{\mathcal{L}_2}

\DeclareMathOperator{\SCP}{SCP}
\DeclareMathOperator{\CCP}{CCP_\varepsilon}
\DeclareMathOperator{\RCP}{RCP}
\newcommand{\Jrcp}{J^\star_{\RCP}}
\newcommand{\Jccp}{J^\star_{\CCP}}
\newcommand{\Jscp}{J^\star_{N}}

\DeclareMathOperator{\CP}{CP}
\DeclareMathOperator{\RP}{RP}
\DeclareMathOperator{\SP}{SP}
\newcommand{\Jrp}{J^\star_{\RP}}
\newcommand{\Jcp}{J^\star_{\CP}}
\newcommand{\Jg}[1]{J^\star_{\RCP_{#1}}}

\DeclareMathOperator{\SCPk}{SCP^{(k)}}
\DeclareMathOperator{\CCPk}{CCP^{(k)}_{\varepsilon_k}}
\DeclareMathOperator{\RCPk}{RCP^{(k)}}

\newcommand{\X}{\mathbb{X}}
\newcommand{\D}{\mathcal{D}}
\newcommand{\C}{\mathcal{C}}
\newcommand{\sat}{\models}
\newcommand{\nsat}{\not\models} 
\newcommand{\lip}{L_{\normalfont{\textbf{SP}}}}
\newcommand{\Lcon}{\mathcal{C}}

\usepackage[colorlinks	= true,
			raiselinks	= true,
			linkcolor	= MidnightBlue,
			citecolor	= Mahogany,
			urlcolor	= ForestGreen,
			pdfauthor	= {Peyman Mohajerin Esfahani},
			pdftitle	= {},
			pdfkeywords	= {},
			pdfsubject	= {Technical Report},
			plainpages	= false]{hyperref}

\usepackage[usenames, dvipsnames]{color}
\usepackage{dsfont,amssymb,amsmath,subfigure, graphicx} 
\usepackage{amsfonts,dsfont,mathtools, mathrsfs,amsthm} 
\usepackage[amssymb, thickqspace]{SIunits}

\usepackage{enumitem}

\allowdisplaybreaks
\date{\today}

\title{Performance Bounds for the Scenario Approach and an Extension to a Class of Non-convex Programs}

\author[P. Mohajerin Esfahani]{Peyman Mohajerin Esfahani}
\author[T. Sutter]{Tobias Sutter}
\author[J. Lygeros]{John Lygeros}
\thanks{The authors are grateful to Diethard Klatte and Stefan Richter for helpful discussions and pointers to references, and Geir Dullerud for motivating the first example.}

\thanks{Research supported by the European Commission under the project MoVeS ({Grant Number 257005}), the HYCON2 Network of Excellence (FP7-ICT-2009-5), and by ETH grant (ETH-15 12-2).}

\thanks{The authors are with the Automatic Control Laboratory, Physikstrasse 3, ETH Z\"urich, 8092 Z\"urich, Switzerland, {\tt \{mohajerin,sutter,lygeros\}@control.ee.ethz.ch} }

\begin{document}
\maketitle

 \begin{abstract}
 	We consider the Scenario Convex Program (SCP) for two classes of optimization problems that are not tractable in general: Robust Convex Programs (RCPs) and Chance-Constrained Programs (CCPs). We establish a probabilistic bridge from the optimal value of SCP to the optimal values of RCP and CCP in which the uncertainty takes values in a general, possibly infinite dimensional, metric space. We then extend our results to a certain class of non-convex problems that includes, for example, binary decision variables. In the process, we also settle a measurability issue for a general class of scenario programs, which to date has been addressed by an assumption. Finally, we demonstrate the applicability of our results on a benchmark problem and a problem in fault detection and isolation. 
 \end{abstract}

\section{Introduction} \label{sec:introduction}

Optimization problems under uncertainty have considerable applications in disciplines ranging from mathematical finance to control engineering. For example most control systems involve some level of uncertainty; the aim of a robust control design is to provide a guaranteed level of performance for all admissible values of the uncertain parameters. In the convex case, two well-known approaches for dealing with such uncertain programs are robust convex programs (RCPs) and chance-constrained programs (CCPs). RCPs consider constraint satisfaction for all, possibly infinitely many, realizations of the uncertainty. While it is known that certain classes of RCPs can be solved as effectively as their non-robust counterparts \cite{ref:Bertsimas-06} in other cases RCPs can be intractable \cite{ref:BenTal-98, ref:BenTal-99, ref:Ghaoui-98, ref:BenTal-01}. For example, the class of parametric linear matrix inequalities, which occur in many control problems, is NP-hard \cite{ref:Boyd-94, ref:Gahinet-96}. CCPs, on the other hand, allow constraint violation with a low probability. The resulting optimization problem, however, is in general non-convex \cite{ref:Prekopa-95, ref:Shapiro-09}.
 
 Computationally tractable approximations to the aforesaid optimization problems can be obtained through the scenario convex programs (SCPs) in which only finitely many uncertainty samples are considered. A natural question in this case is how many samples would be ``enough'' to provide a good solution. To answer this question, one may view the problem from two perspectives: feasibility and objective performance. The literature mainly focuses on the first perspective. In this direction, the authors in \cite{ref:CalCam-05,ref:CalCam-06} initialized a feasibility theory for CCP refined subsequently in \cite{ref:CamGar-08,ref:Cal-10}. They established an explicit probabilistic lower bound for the sample size to guarantee the feasibility of the SCP solutions from a chance-constrained perspective. By contrast, the issue of  performance bounds for both RCP and CCP via SCP has not been settled up to now. \cite{ref:Campi-11} provides a novel perspective in this direction that leads to optimal performance bounds for CCPs. However, it involves the problem of optimal constraint removal, which in general is computationally intractable. 
 
 The first contribution of this article is to address the SCP Performance issue from the objective viewpoint. The key element of our analysis relies on the concept of the worst-case violation inspired by the recent work \cite{ref:KanTak}. The authors of \cite{ref:KanTak} derived an upper bound of the worst-case violation for the SCPs where the uncertainty takes values in a finite dimensional Euclidean space. This result leads to a performance bound for a particular class of $\RCP$s where the uncertainly appears in the objective function, e.g., min-max optimization problems. Motivated by different applications such as control problems with saturation constraints \cite{ref:Campi-09}, fault detection and isolation in dynamical systems \cite{ref:Mohajerin-FDI-TAC}, and approximate dynamic programming \cite{ref:Duf-13}, in this article we first extend this result to infinite dimensional uncertainty spaces. In the sequel, we establish a theoretical bridge from the optimal values of SCP to the optimal values of both RCP and CCP. Along this direction, under mild assumptions on the constraint function (measurability with respect to the uncertainty and lower semicontinuity with respect to the decision variables), we shall also rigorously settle a measurability issue of the SCP optimizer, which to date has been addressed in the literature by an assumption, e.g. \cite{ref:CalCam-06, ref:CamGar-08}. Our second contribution is to extend these results to a class of non-convex programs that, in particular, allows for binary decision variables. In the context of mixed integer programs, the recent work \cite{ref:Calafiore-12} investigates the feasibility perspective of CCPs, which leads to a bound of the required number of scenarios with exponential growth rate in the number of integer variables, whereas our proposed bound scales linearly. 
 
 
 The layout of this article is as follows: In Section \ref{sec:problem} we formally introduce the optimization problems that will be addressed. Our results on probabilistic objective performance for both RCPs and CCPs based on SCPs are reported in Section \ref{sec:perf}. In Section \ref{sec:extension} we extend our results to a class of non-convex programs, including mixed-integer programs with binary variables. To illustrate the proposed methodology, in Section \ref{sec:simulation} the theoretical results are applied to two examples: a benchmark problem whose solution can be computed explicitly, and a fault detection and isolation study with an application to the security of power networks. We conclude in Section \ref{sec:conclusion} with a summary of our work and comment on possible subjects of further research. For better readability, some of the technical proofs and details are given in the appendices. 
 

\section*{Notation} 

 Let $\R_+$ denote the non-negative real numbers. Given a metric space $\D$, its Borel $\sigma$-algebra is denoted by $\borel(\D)$. Throughout this article, measurability always refers to Borel measurability. Given a probability space $\big(\D, \borel(\D), \PP\big)$, we denote the $N$-Cartesian product set of $\D$ by $\D^N$ and the respective product measure by $\PP^N$. An open ball in $\D$ with radius $r$ and center $v$ is denoted by $\ball{v}{r}\Let\{ d\in\D : \| d-v\| <r \}$. The symbol $\sat$ refers to the feasibility satisfaction, i.e., $x \sat \RCP$ means that $x$ is a feasible solution for the program $\RCP$. Similarly, $x \nsat \RCP$ implies that $x$ is not a feasible solution for the optimization problem $\RCP$.

\section{Problem Statement} \label{sec:problem}

 Let $\X \subset \R^n$ be a compact convex set and $c \in \R^n$ a constant vector. Let $\big(\D, \borel(\D), \PP\big)$ be a probability space where $\D$ is a metric space with the respective Borel $\sigma$-algebra $\borel(\D)$. Consider the measurable function $f:\X \times \D \ra \R$, which is convex in the first argument for each $d \in \D$, and bounded in the second argument for each $x \in \X$. We then consider the following optimization problems:
 	\begin{align} 
 	\label{rcp-ccp}
 		\RCP: \left\{ \begin{array}{ll}
 				\min\limits_{x}					& c\tr x   \\
 				\text{s.t. } 			& f(x,d)\leq 0, \quad \forall d \in \D\\
 										& x\in \X 
 		\end{array} \right. ,
 		\qquad
 		\CCP: \left\{ \begin{array}{ll}
 				\min\limits_{x}					& c\tr x   \\
 				\text{s.t. } 			& \PP [ f(x,d) \le 0 ] \ge 1 - \eps \\
 										& x \in \X 
 		\end{array}\right. ,
 	\end{align}
 where $\eps \in [0,1]$ is the constraint violation level for the chance-constrained program. We denote the optimal value of the program $\RCP$ (resp.\ $\CCP$) by $\Jrcp$ (resp.\ $\Jccp$). Suppose $(d_i)_{i=1}^N$ are $N$ independent and identically distributed (i.i.d.) samples drawn according to the probability measure $\PP$. The centerpiece of this study is the scenario program
 	\begin{align}
 	\label{scp}
 		\SCP: \left\{ \begin{array}{ll}
 				\min\limits_{x}					& c\tr x   \\
 				\text{s.t. } 			& f(x,d_i) \leq 0, \quad \forall i \in \{1,\cdots,N\}\\
 										& x\in \X 
 		\end{array} \right. ,
 	\end{align}
 where the optimal solution and optimal value of $\SCP$ are denoted, respectively, by $x^\star_N$ and $\Jscp$. Notice that $\SCP$ is naturally random as it depends on the random samples $(d_i)_{i=1}^N$. 
 
 We assume throughout our subsequent analysis that the following measurability assumption holds, though we shall show in Subsection \ref{subsec:meas} how one may rigorously address this issue without any assumption for a large class of optimization programs (not necessarily convex). 
 	\begin{Ass}
 		\label{a:meas}
 		The $\SCP$ optimizer generates a Borel measurable mapping from $\big(\D^N, \borel(\D^N)\big)$ to $\big(\X, \borel(\X)\big)$ that associates each $(d_i)_{i=1}^N$ with a unique $x^\star_N$.
 	\end{Ass}
 
 The optimization program $\SCP$ in \eqref{scp} is convex and hence tractable even for cases where the problems \eqref{rcp-ccp} are NP-hard. Motivated by this, a natural question is whether there exist theoretical links from $\SCP$ to $\RCP$ and $\CCP$. As mentioned in the introduction, this question can be addressed from two different perspectives: feasibility and objective performance. From the feasibility perspective, we recall the explicit bound of \cite{ref:CamGar-08} which measures the finite sample behavior of $\SCP$:
 
 
 
 	\begin{Thm}[$\CCP$ Feasibility]
 		\label{thm:feas}
 		Let $\beta\in [0,1]$ and $N \ge N(\eps,\beta)$ where 
 		\begin{equation}
 		\label{N}
 			N(\eps,\beta) \Let \min \bigg\{ N\in \N ~\Big|~ \sum_{i=0}^{n-1}  {N \choose i} \eps^{i}(1-\eps)^{N-i}\leq\beta \bigg\}.
 		\end{equation}
 		Then, the optimizer of $\SCP$ is a feasible solution of $\CCP$ with probability at least $1-\beta$. 
  	\end{Thm} 
 
 With ``$\sat$'' notation, the assertion of Theorem \ref{thm:feas} is alternatively stated by $\PP^N\big[ x^\star_N \sat \CCP \big] \ge 1- \beta$, where $\PP^N$ stands for the $N$-fold product probability measure.\footnote{Note that $\PP$ is the probability measure on $\borel(\D)$; for simplicity we slightly abuse the notation, and will be doing so hereinafter. Strictly speaking, one has to define a new probability measure, say $\mathds{Q}$, which is the induced measure on $\borel(\X)$ via the mapping introduced in Assumption \ref{a:meas}.}
 	
 To the best of our knowledge, there is no clear connection between the feasibility of $\RCP$ and the solution of $\SCP$. Furthermore, in Subsection \ref{subsec:feas} we provide an example to challenge the possibility of such a connection. The focus of our study is on the second perspective to seek a (probabilistic) bound for the optimal values $\Jrcp$ and $\Jccp$ in terms of $J^\star_{N}$. 
 

\section{Probabilistic Objective Performance} \label{sec:perf}

 \subsection{Confidence interval for the objective functions}\label{subsec:perf}
 The following definition inspired by the recent work \cite{ref:KanTak} is the key object for our analysis. 
 
 	\begin{Def}
 		\label{def:tail}
 		The \emph{tail probability of the worst-case violation} is the function $p:\X \times \R_+ \ra [0,1]$ defined as
 		\begin{align*}
 			p(x,\delta) \Let \PP\Big[ \sup_{v \in \D} f(x,v) - \delta < f(x,d) \Big].
 		\end{align*}
 		We call $h: [0,1] \ra \R_+$ a \emph{uniform level-set bound} (ULB) of $p$ if for all $\eps \in [0,1]$
 		\begin{align*}
 				h(\eps) \ge \sup \Big \{ \delta \in \R_+ ~\big|~ \inf_{x \in \X} p(x,\delta) \le \eps \Big \}.
 		\end{align*}
 	\end{Def}
 			\begin{figure}[t!]
 				\centering
 				\subfigure[$p(x,\delta)= \PP {[A\cup B]} $]{\label{2D_tail_prob}\includegraphics[scale = 0.6]{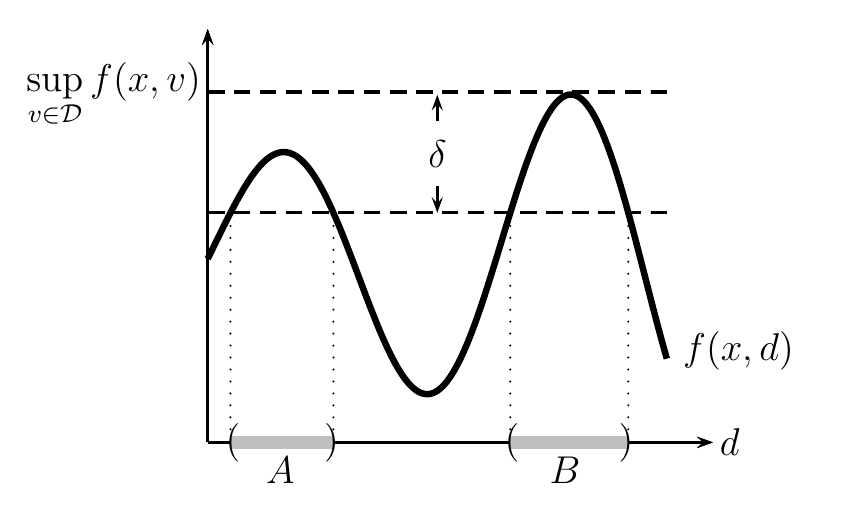}} \quad
 				\subfigure[Tail probability of the worst-case violation]{\label{fig:3D_plot}\includegraphics[scale = 0.6]{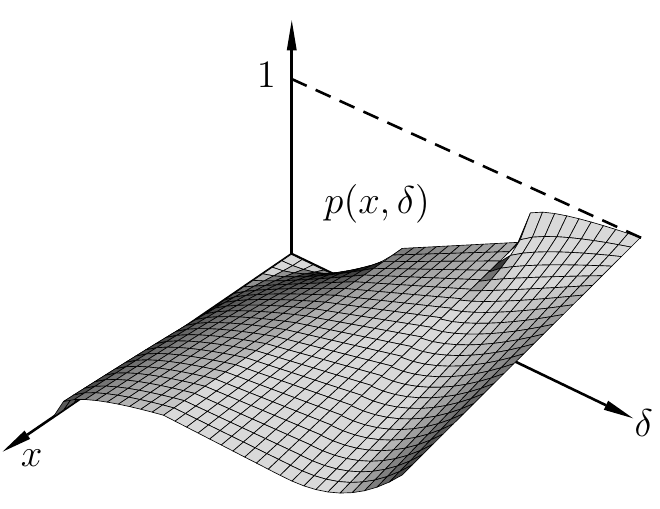}}	\qquad
 				\subfigure[Uniform level set bound]{\label{fig:ULB_plot}\includegraphics[scale = 0.6]{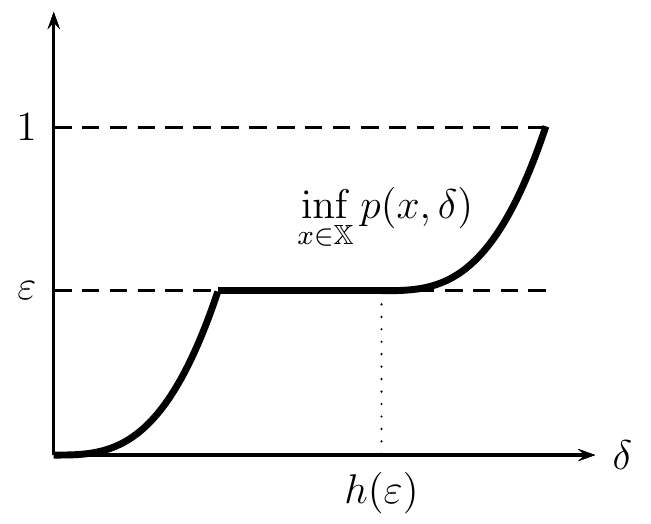}}			
 				\caption{Pictorial representation of Definition \ref{def:tail} }
 				\label{fig:ULB}
 		\end{figure}
 	A pictorial representation of Definition \ref{def:tail} is given in Figure \ref{fig:ULB}. Note that from a statistical perspective the ULB function may be alternatively viewed as an upper bound for the \emph{quantile} function of the $\R$-valued random variable $d \mapsto \big(\sup_{v \in \D}f(x,v)-f(x,d)\big)$ uniformly in decision variable $x \in \X$ (cf. \cite[Section 5.2]{ref:Shao-03}). Proposition \ref{prop:h} at the end of this subsection provides sufficient conditions under which a candidate ULB can be constructed. If the uncertainty set $\D$ is a specific compact subset of a Euclidean space, namely a norm-constrained or more generally a star-shaped set, the authors in \cite{ref:KanTak} provide a constructive approach to obtain an admissible ULB. 
 	
 	
 	Consider the relaxed version of the program $\RCP$ for $\gamma > 0$:
 	\begin{align}
 		\label{rcp-gamma}
 		\RCP_\gamma: \left\{ \begin{array}{ll}
 						\min\limits_{x}					& c\tr x   \\
 						\text{s.t. } 			& f(x,d)\leq \gamma, \quad \forall d \in \D\\
 												& x\in \X 
 		\end{array} \right. ,
 	\end{align}
 	with the optimal value $\Jg{\gamma}$. 
 	
 	\begin{Lem}
 		\label{lem:feas}
 		Let $h:[0,1] \ra \R_+$ be a ULB. Then, 
 			$$ x \sat \CCP \quad \Lra \quad  x \sat \RCP_{h(\eps)}$$
 		that is, the feasible set of the program $\CCP$ with constraint violation level $\eps$ is a subset of the feasible set of the relaxed program $\RCP_\gamma$ with $\gamma \Let h(\eps)$.
 	\end{Lem}
 	\begin{proof}
 		See Appendix \ref{app-A}. 
 	\end{proof}

 	\begin{Ass}[Slater Point]
 		\label{a:duality}
 		There exists an ${x}_0\in\X$ such that $\sup\limits_{d \in \D}f({x}_0,d)<0$.
 	\end{Ass}

 	Under Assumption \ref{a:duality}, we define the constant
 		\begin{align}
 		\label{lip}
 			\lip \Let \frac{\min_{x \in \X}c\tr x - c\tr x_0}{\sup_{d \in \D}f({x}_0,d)}.
 		\end{align}
 	The following lemma is a classical result in perturbation theory of convex programs, which is a significant ingredient for the first result of this article. 
 	\begin{Lem}
 		\label{lem:lip}
 		Consider the relaxed program $\RCP_\gamma$ and its optimal value $\Jg{\gamma}$ as introduced in \eqref{rcp-gamma}. Under Assumption \ref{a:duality}, the mapping $\R_+ \ni \gamma \mapsto \Jg{\gamma} \in \R$ is Lipschitz continuous with constant bounded by $\lip$ in \eqref{lip}, i.e., for all $\gamma_2 \ge \gamma_1\ge 0$ we have
 			$$ 0 \le  \Jg{\gamma_1} - \Jg{\gamma_2} \le \lip(\gamma_2-\gamma_1).$$
 	\end{Lem}
 	
 	\begin{proof}
		See Appendix \ref{app-A}.
 	\end{proof}
 	
  	Assumption \ref{a:duality} requires the existence of a strictly feasible solution $x_0$ which, in general, may not exist. However, in applications where a ``risk-free" decision is available such an assumption is not really restrictive; the portfolio selection problem is an example of this kind \cite{ref:PagCam-12}. In addition, for the class of min-max problems, as a particular case of the program $\RCP$, it is not difficult to see that Assumption \ref{a:duality} holds; see the following remark for more details and Section \ref{ex:2} for an application to the problem of fault detection and isolation.
  	
  	\begin{Rem}[$\lip$ for Min-Max Problems] \label{rem:lip} 
  	In min-max problems, one may inspect that there always exists a Slater point (in the sense of Assumption \ref{a:duality}) with the corresponding constant $\lip$ arbitrarily close to $1$. In fact, it is straightforward to observe that for min-max problems $\Jg{\gamma} = \Jrcp - \gamma$, which readily implies that the Lipschitz constant of Lemma \ref{lem:lip} is $1$. 
  	\end{Rem}
 	
 	The following results are the main contributions of the first part of the article.
 	
 	\begin{Thm}[$\RCP$ Confidence Interval]
 		\label{thm:perf-rcp}
 		Consider the programs $\RCP$ and $\SCP$ in \eqref{rcp-ccp} and \eqref{scp} with the associated optimal values $\Jrcp$ and $\Jscp$, respectively. Suppose Assumption \ref{a:duality} holds and $\lip$ is the constant in \eqref{lip}. Given a ULB $h$ and $\eps, \beta$ in $[0,1]$, for all $N \ge N(\eps,\beta)$ as defined in \eqref{N}, we have
 			\begin{align}
 				\label{J_rcp} 
 				\PP^N \Big[ \Jrcp - \Jscp \in \big[0, I(\eps)\big]\Big] & \ge 1-\beta,
 			\end{align}
 		where 
 			\begin{align}
 				\label{I}
 				I(\eps) \Let \min\Big\{ \lip h(\eps),~ \max_{x \in \X} c\tr x - \min_{x\in\X}c\tr x \Big\}.
 			\end{align}
 	\end{Thm}
 
 	\begin{proof}
 		Due to the definition of the optimization problems $\RCP$ and $\SCP$, the second term of the confidence interval \eqref{I} is a trivial bound. It then suffices to establish the bound for the first term of \eqref{I}. By Theorem \ref{thm:feas}, we know $\PP^N\big[x^\star_N \sat \CCP\big] \ge 1-\beta$ that in view of Lemma \ref{lem:feas} implies
 		\begin{align*}
 		\PP^N\big[x^\star_N \sat \RCP_{h(\eps)}\big] \ge 1-\beta ~\Lra~ \PP^N\big[\Jg{h(\eps)} \le J^\star_N \big] \ge 1-\beta,
 		\end{align*}
 		where $h$ is the ULB, and $\Jg{h(\eps)}$ is the optimal value of the relaxed robust program \eqref{rcp-gamma} with $\gamma \Let h(\eps)$. By virtue of Lemma \ref{lem:lip}, we have $\Jrcp \le \Jg{_{h(\eps)}} + \lip h(\eps)$, that in conjunction with the above implication leads to
 		\begin{align*}
 			\PP^N \Big[ \Jrcp \le J^\star_N + \lip h(\eps) \Big] \ge 1-\beta.
 		\end{align*}
 		 Since the program $\SCP$ is just a restricted version of $\RCP$, it is trivial that $\Jscp \le \Jrcp$ pointwise on $\Omega^N$, which concludes \eqref{J_rcp}.
 		\end{proof} 
 	
 	In accordance with the optimization problem $\CCP$, the following theorem provides similar performance assessment but in both a priori and a posteriori fashions.

 	\begin{Thm}[$\CCP$ Confidence Interval]
 		\label{thm:perf-ccp}
 		Consider the programs $\CCP$ and $\SCP$ in \eqref{rcp-ccp} and \eqref{scp} with the associated optimal values $\Jccp$ and $\Jscp$, respectively. Suppose Assumption \ref{a:duality} holds and $\lip$ is the constant in \eqref{lip}. Let $h$ be a ULB and $\lambda^\star_N$ the dual optimizer of $\SCP$. Given $\beta$ in $[0,1]$, for all $N \ge N(\eps,\beta)$ defined in \eqref{N}, we have 
 			\begin{subequations}
 			\label{J-ccp}
 				\begin{align}
 				\label{J_ccp-prior} \text{A Priori Assessment:}& \qquad 
 				\PP^N \Big[ \Jccp - \Jscp \in \big[-I(\eps), 0\big] \Big] ~\ge 1-\beta,  \\
 				\label{J_ccp-posterior} \text{A Posteriori Assessment:}& \qquad 
 				\PP^N \Big[ \Jccp - \Jscp \in \big[-I_N(\eps), 0\big] \Big]  \ge 1-\beta,
 				\end{align}
 			\end{subequations}
 		where the a priori interval $I(\eps)$ is defined as in \eqref{I}, and the a posteriori interval is 
 			\begin{align}
 			\label{I_N}
 				I_N(\eps) \Let \min\Big\{ \big\|\lambda^\star_N \big\|_1 h(\eps),~ \max_{x \in \X} c\tr x - \min_{x\in\X}c\tr x \Big\}.
 			\end{align}
 	\end{Thm}
 		
 	\begin{proof}
 		Similar to the proof of Theorem \ref{thm:perf-rcp}, we only need to show the first term of the confidence interval \eqref{I_N}. In light of Theorem \ref{thm:feas} and Lemma \ref{lem:feas}, we know that 
 			\begin{align}
 			\label{e:thm-ccp:1}
 				\PP^N \big[ \Jg{h(\eps)}\le \Jccp \le \Jscp \big] \ge 1 - \beta.
 			\end{align}
 		In the same spirit as the previous proof, Lemma \ref{lem:lip} ensures $\Jscp \le \Jrcp \le \Jg{h(\eps)} + \lip h(\eps)$ everywhere on $\Omega^N$, which together with \eqref{e:thm-ccp:1} arrives at \eqref{J_ccp-prior}. 
 		
 		To show \eqref{J_ccp-posterior}, let us consider the scenario counterpart of the relaxed program $\RCP_\gamma$ in \eqref{rcp-gamma} with $\gamma \Let h(\eps)$. We denote the optimal value of this scenario program by $J^\star_{N,h(\eps)}$. Thus, we have $J^\star_{N,h(\eps)} \le \Jg{h(\eps)}$ with probability 1. Notice that Assumption \ref{a:duality} also holds for the scenario program $\SCP$, and consequently Lemma \ref{lem:lip} is applicable to $\SCP$ as well. In fact, following the proof of Lemma \ref{lem:lip} \cite[p.~250]{ref:Boyd-04}, one can infer that the Lipschitz constant of the perturbation function can be over approximated by the $\ell_1$-norm of a dual optimizer of the optimization program. Therefore, applying Lemma \ref{lem:lip} to $\SCP$ yields to $\Jscp -\|\lambda^\star_N \|_1 h(\eps) \le J^\star_{N,h(\eps)} \le \Jg{h(\eps)}$ pointwise on $\Omega^N$. Substituting into \eqref{e:thm-ccp:1} leads to \eqref{J_ccp-posterior}.		
 	\end{proof}

	The parameter $\eps$ in Theorem \ref{thm:perf-rcp} is a design choice which can be tuned to shrink the confidence interval $[0,I(\eps)]$. On the contrary, in Theorem \ref{thm:perf-ccp} the parameter $\eps$ is part of the problem data associated with the program $\CCP$. That is, in Theorem \ref{thm:perf-ccp} $I(\eps)$ is indeed fixed and the number of scenarios $N$ in $\SCP$ only improves the confidence level $\beta$. In a same spirit but along a different approach, \cite[Theorem 6.1]{ref:Cal-10} bounds $\Jscp$ by the optimal solutions of two chance-constrained programs associated with different constraint violation levels, say $\bar \eps < \eps$. This value gap between the chance-constrained program and its scenario counterpart (either as explicitly derived in Theorem \ref{thm:perf-ccp} or implicitly by two chance-constrained programs in \cite[Theorem 6.1]{ref:Cal-10}) represents an inherent difference. To arbitrarily reduce the gap for $\CCP$, one may resort to \emph{optimally} discarding a fraction of scenarios, which is in general computationally intractable; see for example \cite[Theorem 6.1]{ref:Campi-11} and \cite[Theorem 6.2]{ref:Cal-10}. 
 	
	By virtue of Theorem \ref{thm:perf-rcp}, the gap between $\Jrcp$ and $\Jscp$ is effectively quantified by a ULB $h(\eps)$ as introduced in Definition \ref{def:tail}. To control the behavior of $h(\eps)$ as $\eps \ra 0$, one may require more structure on the measure $\PP$ defined on $\big(\D,\borel(\D) \big)$. Proposition \ref{prop:h} addresses this issue by introducing sufficient conditions concerning the measure of open balls in $\borel(\D)$ and the continuity of the constraint mapping in the uncertainty argument. 
 	 
 
 	\begin{Prop}
 		\label{prop:h}
 		Assume that the mapping  $\D \ni d \mapsto f(x,d) \in \R$ is Lipschitz continuous with constant $L_d$ uniformly in $x \in \X$. Suppose there exists a strictly increasing function $g:\R_+ \ra [0,1]$ such that 
 		$$ \PP \big[ \ball{d}{r} \big] \ge g(r), \qquad \forall d \in \D, $$
 		where $\ball{d}{r} \subset \D$ is an open ball centered at $d$ with radius $r$. Then, $h(\eps) \Let L_d\ g^{-1}(\eps)$ is a ULB in the sense of Definition \ref{def:tail}, where $g^{-1}$ is the inverse function of $g$.
 	\end{Prop}
 	
 	\begin{proof}
 		See Appendix \ref{app-A}. 
 	\end{proof}
 	
 	Proposition \ref{prop:h} generalizes the corresponding result of \cite[Lemma 3.1]{ref:KanTak} by allowing the uncertainty space $\D$ to be possibly an infinite dimensional space. Note that the required assumptions in Proposition \ref{prop:h} implicitly require $\D$ to be bounded, though in practice this may not be really restrictive. 
 	
	\begin{Rem} \label{rem:curse}
	Two remarks regarding the function $g$ in Proposition \ref{prop:h} are in order:

		\begin{enumerate}[label = (\roman*), align = right, widest=m, leftmargin= 1cm]
			\item \textbf{Explicit expression:} Under the hypotheses of Proposition \ref{prop:h}, Theorem \ref{thm:perf-rcp} can be expressed in more explicit form. Let $\eps$ and $\beta$ be in $[0,1]$, $L_d$ be the Lipschitz constant of the constraint function $f$ in $d$, $\lip$ be the constant \eqref{lip}, and $N(\cdot,\cdot)$ be as defined in \eqref{N}. Then, for any $N \ge N\big( g(\frac{\eps}{\lip L_d}),\beta \big)$ we have
			\begin{align*}
				\PP^N \Big[ \Jrcp - \Jscp \in \big[0,\eps \big] \Big] \ge 1-\beta.
			\end{align*}
			
			\item \textbf{Curse of dimensionality:} For an $n_d$-dimensional uncertainty set $\D$, the number of disjoint balls in $\D$ with radius $r$ grows proportional to $r^{-n_d}$ as $r$ decreases. Thus, the assumptions of Proposition \ref{prop:h} imply that $g(r)$ is of the order of $r^{n_d}$. Therefore, for the desired precision $\eps$, as detailed in the preceding remark, the required number of samples $N$ grows exponentially as ${\eps}^{-n_d}$. This appears to be an inherent feature when one seeks to bound the optimal value via scenario programs; see \cite{ref:LecLyg-08,ref:LecLyg-10} for similar observations.
			
		\end{enumerate}
	\end{Rem}

%
 
 \subsection{Feasibility of \textnormal{RCP} via \textnormal{SCP}}\label{subsec:feas}
 
 	In this subsection we provide an example to show the inherent difficulty of the feasibility connection from $\SCP$ to the original problem $\RCP$. Consider the following $\RCP$ with its $\SCP$ counterpart in which both decision and uncertainty space are compact subsets of $\R$:
 		\begin{align*}
 			\left\{ \begin{array}{ll}
 				\min\limits_{x}					& -x   \\
 				\text{s.t. } 			& x - d \leq 0, \quad \forall d \in \D \Let [0,1]\\
 										& x\in \X\Let [-1,1] 
 		\end{array} \right.
 		\qquad 
 			\left\{ \begin{array}{ll}
 				\min\limits_{x}					& -x   \\
 				\text{s.t. } 			& x - d_i \leq 0, \quad \forall i \in \{1,\cdots,N\}\\
 										& x\in \X\Let [-1,1] 
 			\end{array} \right. .
 	\end{align*}	
 	It is not difficult to see that the feasible set of the robust program is $[-1,0]$ with the optimizer $x^\star = 0$, whereas the optimizer of its scenario program is $x^\star_N = \min_{ i \le N} d_i$. If the probability measure $\PP$ does not have atoms (point measure), we have $\PP^N \big[ \min_{ i \le N} d_i > 0 \big] = 1$. Thus, one can deduce that
 		 \begin{align*}
 		 	\PP^N \big[ x^\star_N \sat \RCP \big] = 0, \qquad \forall \PP\in\mathcal{P}, \quad \forall N \in \N,
 		 \end{align*} 
 	where $\mathcal{P}$ is the family of all nonatomic measures on $\big(\D,\borel(\D)\big)$. More generally, if the set $\arg\max_{d \in \D}f(x,d)$ has measure zero for any $x \sat \RCP$ (e.g., when $f$ is convex in $d$ and the boundary of $\D$ has zero measure), then the program $\SCP$ will almost surely return infeasible solutions to the program $\RCP$, as the worst-case scenarios are almost surely neglected.
 	
 \subsection{Measurability of the \textnormal{SCP} optimizer}\label{subsec:meas}
 	The objective of this subsection is to address the standing Assumption \ref{a:meas}. The measurability of the optimizer $x^\star_N$ for the scenario program $\SCP$ is a rather involved technical issue. In fact, to the best of our knowledge, in the literature this issue is always resolved by introducing an assumption. Let us highlight that the measurability of optimal values and the set of optimizers as well as the existence of a measurable selection are classical results in this context, see for instance  \cite[Theorem 14.37, p.\ 664]{ref:Rockafellar-10}. However, there is no a priori guarantee that the obtained optimizer of the program $\SCP$ can be viewed as a measurable mapping from ${\D}^N$ to $\X$. Toward this issue, we propose a ``two-stage" optimization program, in the \emph{lexicographic} sense in the context of multi-objective optimization problems \cite{ref:Lex-04}, in which the measurability of this mapping is ensured for a large class of programs (not necessarily convex).
 	
 	For the rest of this section we assume that $\X \subset \R^n$ is closed and the mapping $x \mapsto f(x,d)$ is lower semicontinuous. Consider the scenario program $\SCP$ as defined in \eqref{scp} with the corresponding optimal value $\Jscp$; $\SCP$ is assumed to be feasible with probability one. Given the same uncertainty samples $(d_i)_{i=1}^N$ as in $\SCP$, we introduce the second program
 	\begin{align}
 	 \label{scp-2}
 		\left\{ \begin{array}{ll}
 				\min\limits_{x}					& \phi(x)   \\
 				\text{s.t. } 			& f(x,d_i) \leq 0, \quad \forall i \in \{1,\cdots,N\}\\
 										& c\tr x \le \Jscp \\
 										& x\in \X 
 		\end{array} \right. ,
 	\end{align}
 	where $\phi : \R^n \ra \R$ is a strictly convex function. Let us denote the optimizer of the above program  by $\wt{x}^\star_N$. It is straightforward to observe that $\wt{x}^\star_N$ is indeed an optimizer of the program $\SCP$. 
 	
 	\begin{Prop}[Measurability of the Optimizer]
 		\label{prop:meas}
 		Consider the sequential two-stage programs $\SCP$ and \eqref{scp-2}, with the optimizer $\wt{x}^\star_N$ for the latter program. Then, the mapping $\D^N \ni (d_i)_{i=1}^N \mapsto \wt{x}^\star_N \in \X$ is a singleton and measurable.
 	\end{Prop}
 	
 	\begin{proof}
 		See Appendix \ref{app-A} along with some preparatory lemmas. 
 	\end{proof}
 	
 	The above two-stage program may be viewed as a tie-break rule \cite{ref:CalCam-05} or a regularization procedure \cite[Section 2.1]{ref:Cal-10}, which was proposed to resolve the uniqueness property of the $\SCP$ optimizer. Proposition \ref{prop:meas} indeed asserts that the same trick ensures the measurability of the optimizer as well. 
 	
 	\begin{Rem}[{Measurability of the Feasible Set}]
		The measurability of the feasibility event $\wt x^\star_N \sat \CCP$ (equivalently the measurability of the mapping $x \mapsto \PP[f(x,d) \le 0 ]$) is a straightforward consequence of Proposition \ref{prop:meas} and Fubini's Theorem \cite[Thm.\ 18.3, p.\ 234]{ref:Bil-95}.
 	\end{Rem}
	
\section{Extension to a Class of Non-Convex Programs} \label{sec:extension}

This section extends the results developed in Section \ref{subsec:perf} to a class of non-convex problems. Consider a family of programs introduced in \eqref{rcp-ccp} in which the program data are indexed by $k$, i.e., $(\X_k, f_k, \eps_k)_{k=1}^m$. We assume that each tuple $(\X_k, f_k, \eps_k)$ satisfies the required conditions in Section \ref{sec:problem} (i.e., $\X_k$ is a compact convex set and the mapping $x \mapsto f_k(x,d)$ is convex for every $d \in \D$), and the corresponding programs are denoted by $\RCPk$ and $\CCPk$ as defined in \eqref{rcp-ccp}. Consider the following (non-convex) optimization problems:
 	\begin{align}
 	\label{rp-cp}
 		& \RP: \left\{ \begin{array}{ll}
 				\min\limits_{x}					& c\tr x   \\
 				\text{s.t. } 			& x \sat \bigcup\limits_{k=1}^m \RCPk
 		\end{array} \right. 
 		\qquad
 		\CP: \left\{ \begin{array}{ll}
 				\min\limits_{x}					& c\tr x   \\
 				\text{s.t. } 			& x \sat \bigcup\limits_{k=1}^m \CCPk
 		\end{array}\right.,
 	\end{align}
 where $x \sat \bigcup_{k=1}^m \RCPk$ \big(resp. $x \sat \bigcup_{k=1}^m \CCPk$\big) indicates that there exists $k \in \{1,\cdots,m\}$ such that $x \sat \RCPk$ \big(resp. $x \sat \CCPk$\big). In other words, the programs \eqref{rp-cp} seek an optimal solution which is feasible for  at least one of the subprograms indexed by $k$, while the uncertainty space $\D$ as well as the associated measure $\PP$ is shared between all the subprograms. Similarly, given i.i.d.\ samples $(d_i)_{i=1}^{N} \subset \D$ with respect to the probability measure $\PP$, consider the scenario (non-convex) program
 	\begin{align}
 	\label{sp}
 		\SP: \left\{ \begin{array}{ll}
 						\min\limits_{x}					& c\tr x   \\
 						\text{s.t. } 			& x \sat \bigcup\limits_{k=1}^m \SCPk
 		\end{array} \right. .
 	\end{align}
 Each subprogram $\SCPk$ is defined according to the scenario convex program \eqref{scp} associated with the program data $(\X_k, f_k)$ while the uncertainty samples $(d_i)_{i=1}^{N}$ are the same for all $k \in \{1,\cdots,m\}$. Before proceeding with the main result of this section, let us point out that the programs \eqref{rp-cp} contain, for example, a class of mixed integer programs. Let $f: \R^n \times \{0,1\}^{\ell} \times \D \ra \R$ be the constraint function in \eqref{rcp-ccp}. It is straightforward to see that a chance-constrained mixed integer program can be formulated as
	\begin{align*}
		& \left\{ \begin{array}{ll}
				\min\limits_{x,y}					& c\tr x   \\
 				\text{s.t. } 			& \PP [ f(x,y,d) \le 0 ] \ge 1 - \eps \\
 										& x \in \X, \quad y \in \{0,1\}^\ell 
		\end{array} \right. 
		\Llra \quad 
		\left\{ \begin{array}{ll}
				\min\limits_{x}					& c\tr x   \\
 				\text{s.t. } 			& \max\limits_{k \in \{1,\cdots,2^\ell\}} \PP [ f_k(x,d) \le 0 ] \ge 1 - \eps \\
 										& x \in \X 
		\end{array}\right.,
	\end{align*}
 where $f_k(x,d) \Let f(x,y_k,d)$ for each selection of the binary variables $y_k \in \{0,1\}^\ell$. Then, by setting $m \Let 2 ^\ell$, $\X_k \Let \X$, $\eps_k \Let \eps$, the right-hand side of the above relation is readily in the framework of \eqref{rp-cp}. A similar argument also holds for the robust mixed integer problems counterparts. 
 
 As a first step, we extend the feasibility result of Theorem \ref{thm:feas} to the non-convex setting in \eqref{rp-cp}.
 	\begin{Thm}[$\CP$ Feasibility]
 		\label{thm:feas-nonconvex}
 		Let $\vec{\eps} \Let (\eps_1,\cdots,\eps_m)$, $\beta\in(0,1]$, and $N \ge \wt{N}(\vec \eps,\beta)$ where 
 		\begin{equation}
 		\label{N-tilde}
 			\wt N(\vec{\eps},\beta) \Let \min \bigg\{ N\in \N ~\Big|~ \sum_{k=1}^{m} \sum_{i=0}^{n-1}{N\choose i} \eps_k^{i}(1-\eps_k)^{N-i}\leq\beta \bigg\}.
 		\end{equation}
 		Then, the optimizer of $\SP$ is a feasible solution of $\CP$ with probability at least $1-\beta$. 
  	\end{Thm} 
 
 \begin{proof}
 	Let $ x^\star_{N,k}$ be the optimizer of the subprogram $\SCPk$. By virtue of Theorem \ref{thm:feas}, one can infer that 
 	\begin{align*}
 		\PP^N\Big[ x^\star_{N,k} \nsat \CCPk \Big] < \sum_{i=0}^{n-1}{N\choose i} \eps_k^{i}(1-\eps_k)^{N-i}.
 	\end{align*}
 	On the other hand, it is straightforward to observe that the optimizer of the program $\SP$, denoted by $ x^\star_N$, belongs to the set $(x^\star_{N,k})_{k=1}^m$. Therefore,
 	\begin{align*}
 		\PP^N\Big[ x^\star_N \nsat \CP  \Big] &\le \PP^N\Big[ \exists k \in \{1,\cdots,m\} ~\big|~ x^\star_{N,k} \nsat \CCPk  \Big] \le \sum_{k=1}^m \PP^N\Big[ x^\star_{N,k} \nsat \CCPk  \Big] \\
 		& < \sum_{k=1}^m \sum_{i=0}^{n-1}{N\choose i}\eps_k^{i}(1-\eps_k)^{N-i}, 
 	\end{align*}
 	leading to the desired assertion.
 \end{proof}
 
 		\begin{Rem}[Growth rate]
 			\label{rem:}
 			Notice that the number of subprograms, $m$, contributes to the confidence level $\beta$ in a linear fashion. As an illustration, suppose  $\eps_k \Let \eps$. In this case, one can easily verify that the confidence level of the non-convex program $\SP$ can be set equal to $\frac{\beta}{m}$, where $\beta$ is the confidence level of each of the subprograms $\SCPk$. From a computational perspective, one can follow the same calculation as in \cite{ref:Cal-09}, and deduce that the contribution of $m$ to the number of the required samples $\wt N$ appears in a logarithm. Thus, in our example of mixed integer programming above, the required number of samples grows linearly in the number of binary variables, which for most of applications could be considered a reasonable growth rate. 
 		\end{Rem}
 	
	The literature on computational schemes on non-convex problems is mainly based on statistical learning methods. A recent example of this nature is \cite{ref:Alamo-09}, which considers a class of problems involving Boolean expressions of polynomial functions. Given the degree and number of polynomial functions ($\alpha$ and $k$, respectively), the explicit sample bounds of \cite{ref:Alamo-09} scale with ${\eps^{-1}} \log (\alpha k \eps^{-1})$ as opposed to our result in \eqref{N-tilde} which grows proportional to $\eps^{-1}\log(m)$. We now proceed to extend the main results of Subsection \ref{subsec:perf}, i.e., Theorems \ref{thm:perf-rcp} and \ref{thm:perf-ccp}, to the non-convex settings \eqref{rp-cp} and \eqref{sp} at once. 
	

 	\begin{Thm}[$\RP$ \& $\CP$ Confidence Intervals]
 		\label{thm:perf-nonconvex}
 		Consider the programs $\RP$, $\CP$, and $\SP$ in \eqref{rp-cp} and \eqref{sp} with the corresponding optimal values $\Jrp$, $\Jcp$, and $\Jscp$. Given $k \in \{1,\cdots,m\}$ and the program data $(\X_k , f_k)$, let Assumption \ref{a:duality} hold and $I^{(k)}$ and $I_N^{(k)}$ be the a priori and a posteriori confidence intervals of the $k^{\text{th}}$ subprogram as defined in \eqref{I} and \eqref{I_N}. 
 		Then, given $\beta \in [0,1]$ and $\vec{\eps} \Let (\eps_1, \cdots, \eps_m) \in [0,1]^m$, for all $N \ge \wt N(\vec{\eps}, \beta)$ as defined in \eqref{N-tilde} we have
 						\begin{align*}
 						\text{A Priori Assessment:}& \qquad 
 						\begin{cases}
 							\PP^N \Big[ \Jrp - \Jscp \in \big[0, ~\max\limits_{k\le m} I^{(k)}(\eps)\big] \Big] &\ge 1-\beta, \\
 							\PP^N \Big[ \Jcp - \Jscp \in \big[-\max\limits_{k\le m} I^{(k)}(\eps), ~0\big] \Big]  &\ge 1-\beta,  
 						\end{cases} 
 						\\
 						\text{A Posteriori Assessment:}& \qquad \quad \PP^N \Big[ \Jcp - \Jscp \in \big[-\max\limits_{k\le m} I^{(k)}_N(\eps), ~0\big] \Big] \quad  \ge 1-\beta  .
 						\end{align*}
 	\end{Thm}
 	
 	\begin{proof}[Sketch of the proof]
 		The proof effectively follows the same lines as in the proofs of Theorems \ref{thm:perf-rcp} and \ref{thm:perf-ccp}. To adapt the required preliminaries, let us recall again that the optimizer of the programs \eqref{rp-cp} is one of the optimizers of the respective subprograms. The same assertion holds for the random program \eqref{sp} as well. Moreover, since each subprogram of \eqref{sp} fulfills the assumptions of Subsection \ref{subsec:perf}, Lemmas \ref{lem:feas} and \ref{lem:lip} also hold for each subprogram with the corresponding data $(\X_k , f_k)$. Therefore, in light of Theorem \ref{thm:feas-nonconvex}, it only suffices to consider the worst-case possibility among all the subprograms.
 	\end{proof}


\section{Simulation Results} \label{sec:simulation}
	
	This section presents two examples to illustrate the theoretical results developed in the preceding sections and their performance. We first apply the results to a simple example whose analytical solution is available.
 
 \subsection{Example 1: Quadratic Constraint via Infinite Hyperplanes}  \label{ex:1}
 	Let $x = [x_1, x_2]\tr$ be the decision variables selected in the compact set $\X \Let [0,1]^2 \subset \R^2$, the linear objective function defined by $c \Let [-1,-1]\tr$, and the constraint function $f(x,d) \Let x_1\cos(d) + x_2\sin(d) - 1$ where the uncertainty $d$ comes from the set $\D \Let [0,2\pi]$. Consider the optimization problems introduced in \eqref{rcp-ccp} where $\PP$ is the uniform probability measure on $\D$. It is not difficult to infer that the infinitely many hyperplane constraints can be replaced by a simple quadratic constraint. That is, for any $\gamma \ge 0$
 		\begin{align*}
 			\max_{d \in [0,2\pi]} x_1\cos(d) + x_2\sin(d) - 1 \le \gamma \qquad \Longleftrightarrow \qquad x_{1}^{2}+x_{2}^{2} \le (\gamma + 1)^2.
 		\end{align*}
 	In the light of the above observation, we have the analytical solutions
 		\begin{align}
 		\label{solu}
 			\Jg{\gamma} = \max \Big\{-\sqrt{2}(\gamma + 1), -2 \Big\}, \qquad \Jccp = \max \Big\{ \frac{-\sqrt{2}}{\cos(\pi\eps)},-2 \Big\},
 		\end{align}
 	where $\Jg{\gamma}$ and $\Jccp$ are the optimal values of the optimization problems $\RCP_\gamma$ and $\CCP$ as defined in \eqref{rcp-gamma} and \eqref{rcp-ccp}, respectively. The pictorial representation of the solutions is in Figure \ref{fig:CCP}. 
 	\begin{figure}[t!] 
 			 \begin{minipage}{0.5\linewidth} 
 			  \centering 
 			  \includegraphics[scale = 0.82]{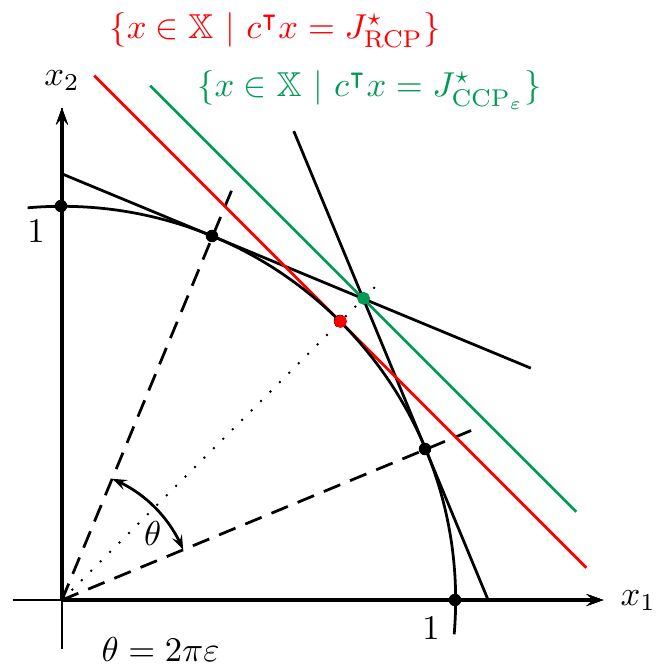} 
 			  \caption{Analytical solutions of Example~1} 
 			  \label{fig:CCP} 
 			 \end{minipage} 
 			 \begin{minipage}{0.49\linewidth} 
 			  \centering 
 			  \includegraphics[scale = 0.25]{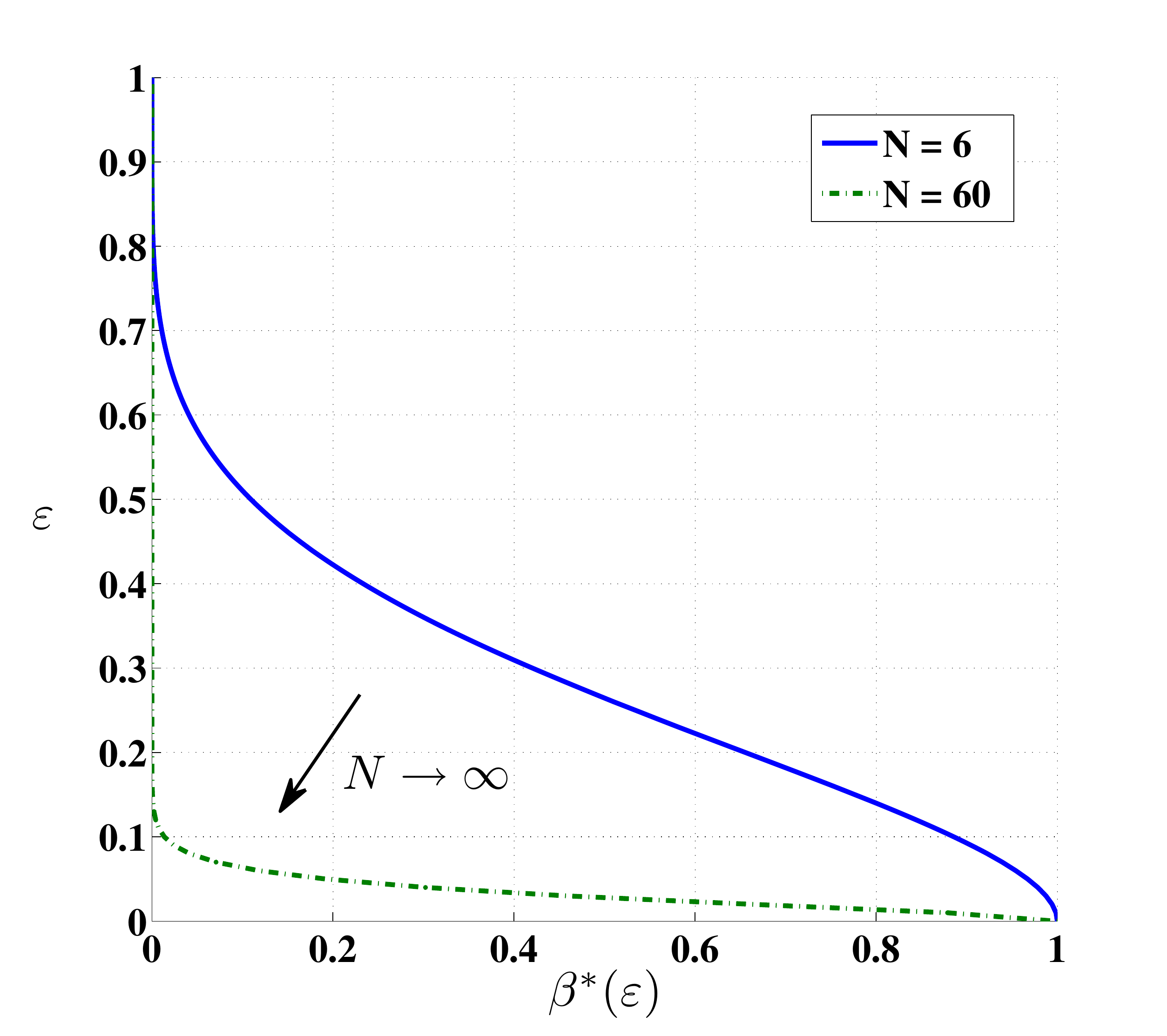}
 			  \caption{Behavior of the confidence level $\beta^*(\eps)$ in terms of scenarios numbers $N$} 
 			  \label{fig:EpsBeta} 
 			 \end{minipage} 
 	\end{figure}
 	
 	Let us fix the number of scenarios $N$ for $\SCP$ in \eqref{scp} with the optimal value $\Jscp$. Given $N$ and $\eps \in [0,1]$, the confidence level $\beta \in [0,1]$ associated with our theoretical results is
 		\begin{align*}
 			\beta^*(\eps) \Let \sum_{i=0}^{n-1} {N\choose i} \eps^i(1-\eps)^{N-i} = (1 - \eps)^N + N\eps(1-\eps)^{N-1},
 		\end{align*} 
 	where $n = 2$ in this example. Figure \ref{fig:EpsBeta} depicts the behavior of $\beta^*(\eps)$ for different values of $N$. Note that $x_0=[0, 0]\tr$ is a Slater point in the sense of Assumption \ref{a:duality} with the corresponding constant $\lip \Let \frac{-2-0}{-1}=2$ (cf. \eqref{lip}). Moreover, it is easy to see that the mapping $d \mapsto f(x,d)$ has the Lipschitz constant $L_d = \sqrt{2}$ over the compact set $\X = [0,1]^2$.
	Thanks to Proposition \ref{prop:h} (and periodicity of the constraint function over the interval $[0,2\pi]$), it is straightforward to introduce $g(r) = \frac{r}{\pi}$, and consequently obtain the ULB candidate $h(\eps) \Let \sqrt{2}\pi \eps$. Then, the confidence interval defined in \eqref{I} is $I(\eps) \Let \max\{2\sqrt{2}\pi \eps, 2\}$. As shown in Theorem \ref{thm:perf-rcp} (resp.\ Theorem \ref{thm:perf-ccp}) we know that $\Jrcp - \Jscp \in [0,I(\eps)]$ \big(resp.\ $\Jccp -\Jscp \in [-I(\eps),0]$\big) with probability at least $1-\beta^*(\eps)$ for any $\eps \in [0,1]$. To validate this result, we solve the program $\SCP$ for $M$ different experiments. For each experiment $k \in \{1,\cdots,M\}$, we draw $N$ scenarios $\big(d_{i}(k)\big)_{i=1}^N \subset [0,2\pi]$ with respect to the uniform probability distribution $\PP$ and solve the program $\SCP$. Let $\Jscp(k)$ be the optimal value of the $k^{\text{th}}$ experiment. Given $\beta \in [0,1]$, the empirical confidence interval of the program $\RCP$ can be represented by the minimal $\wt I (\beta)$ so that the interval $[0,\wt I(\beta)]$ contains $\Jscp(m) - \Jrcp$ for at least $m$ experiments where $\frac{m}{M} \ge 1 - \beta$, i.e., 
 		\begin{align*}
 			\wt I(\beta) \Let \min \Big\{\wt I \in \R_{+} & ~\big|~ \exists A \subset \{1,\cdots, M\} ~:~ \\
 			& |A|  \ge (1-\beta)M \quad \mbox{and} \quad \Jrcp - \Jscp(k) \in [0,\wt I] ~ \forall k \in A \Big\}.
 		\end{align*}
 	Regarding the program $\CCP$, notice that the empirical confidence interval depends on both parameters $\eps$ and $\beta$ since the analytical optimal values $\Jccp$ depends on $\eps$ as well. Hence, we define 
 		\begin{align*}
 			\wt I_{\eps}(\beta) \Let \min \Big\{\wt I \in \R_{+} & ~\big|~ \exists A \subset \{1,\cdots, M\} ~:~ \\
 			& |A|  \ge (1-\beta)M \quad \text{and} \quad \Jccp - \Jscp(k) \in [-\wt I, 0] ~ \forall k \in A \Big\}.
 		\end{align*}
 		
		\begin{figure}[th!]
			\centering
			\subfigure[Simulations for $N = 6$]{\label{fig:N6}\includegraphics[scale = 0.22]{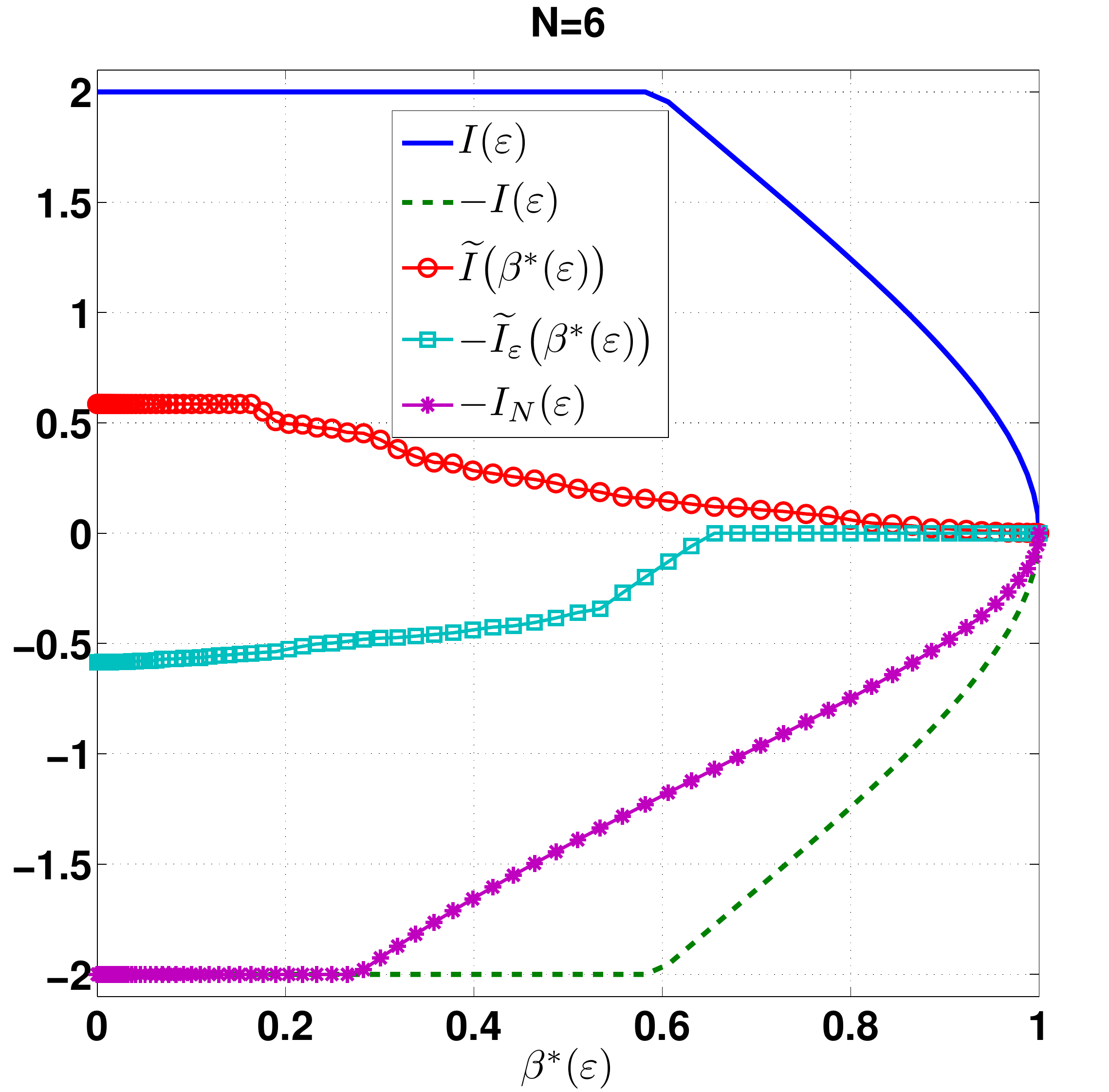}} \qquad
 			\subfigure[Simulations for $N = 60$]{\label{fig:N60}\includegraphics[scale = 0.22]{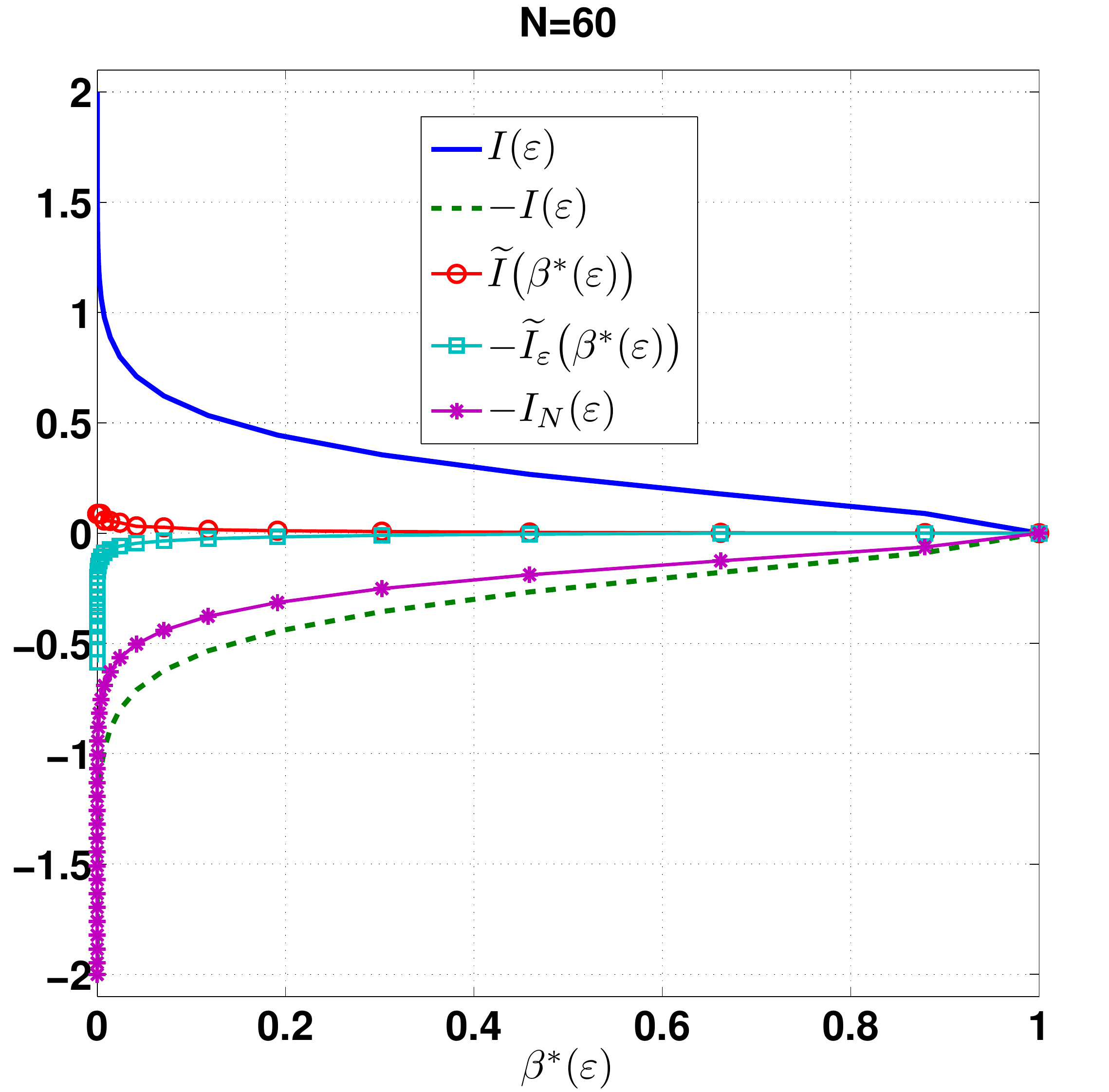}}
 			\caption{Numerical results for Example 1}
 			\label{fig:ex1}
 		\end{figure}
 	
 	The sets $\wt I(\beta)$ and $\wt I_{\eps}(\beta)$ are in close relation with sample quantiles in the sense of \cite[Section 5.3.1]{ref:Shao-03}. In the following simulations the number of experiments is set to $M = 2000$. Figures \ref{fig:N6} and \ref{fig:N60} depict our theoretical performance bound $I(\eps)$ for $N = 6$ and $N = 60$ in comparison with the empirical bounds $\wt I\big(\beta^*(\eps)\big)$ and $\wt I_\eps\big(\beta^*(\eps)\big)$ where $\beta^*(\eps)$ is the confidence level in Figure \ref{fig:EpsBeta}. As our theoretical results suggest, the confidence interval $[0,I(\eps)]$ (resp. $[-I(\eps),0]$) contains the empirical interval $\big[0, \wt I\big(\beta^*(\eps)\big) \big]$ (resp.\ $\big[-\wt I_\eps\big(\beta^*(\eps)\big), 0\big]$). Moreover, to demonstrate the a posteriori confidence interval in Theorem \ref{thm:perf-ccp}, we choose one of the experiments and depict the corresponding confidence interval $ I_N(\eps)$ versus $\beta^*(\eps)$ as well. Note that in both cases of Figure \ref{fig:ex1} the a posteriori confidence interval proposes a tighter bound than the a priori confidence interval. With this observation, we conjecture that in general the dual optimizer of $\SCP$ may happen to be a better approximation in comparison with the constant $\lip$ introduced in \eqref{lip}. 
 	

 \subsection{Example 2: Fault Detection and Isolation}\label{ex:2}
 	The task of fault detection and isolation (FDI) involves generating a diagnosis signal to detect the occurrence of a specific fault. This is typically accomplished by designing a filter with all available signals as inputs (e.g., control signals and given measurements) and a scalar output that implements a non-zero mapping from the fault to the residual while decoupling unknown disturbances. In \cite{ref:Mohajerin_CDC12}, a scalable optimization based approach is proposed to design an FDI filter for a class of nonlinear differential algebraic equation (DAE) where the filter is trained for finite number of disturbance signatures. The class of disturbances is further extended to a probability space in \cite{ref:Mohajerin-FDI-TAC} where the filter performance is quantified in a probabilistic fashion.
 
 	As a particular subclass of DAEs, consider the nonlinear differential equation 
 	\begin{align}\label{ode}
 		\begin{cases}
 		\dot{X}(t) = E\big(X(t)\big)+ AX(t) + B_d d(t) + B_f f(t)\\
 		Y(t) = CX(t)
 		\end{cases},
 	\end{align}
 	where the matrices $A, B_d, B_f, C$ and the function $E(\cdot)$ describe the linear and nonlinear dynamics of the model, respectively. Following \cite{ref:Mohajerin_CDC12,ref:Mohajerin-FDI-TAC}, we restrict the class of filters to linear transfer functions whose residual consists of two terms: $r = G[x](f) + r[x](d)$ where $G[x]$ is a linear time invariant transfer function expressing the mapping from the fault $f(\cdot)$ to the residual, and $r[x](d)$ is the contribution of the unknown disturbance $d(\cdot)$, and $x \in \R^n$ denotes the coefficients of the FDI filter to be designed. For linear systems (i.e., $E \equiv 0$) perfect decoupling between $d$ and $r$ may be possible (i.e., $r[x](d)\equiv 0$ for all $d$). For nonlinear systems, however, may not be the case. In this light, to minimize the impact of nonlinearities and disturbances on the residual, an optimal FDI filter can be obtained by the min-max program 
 		\begin{align}
 		\label{fdi}
 				\left\{ \begin{array}{ll}
 					\min\limits_{x, \gamma}	& \gamma \\
 			\text{s.t. } 	& x\tr Q_d x \le \gamma, \quad \forall d \in \D \\
 							& H x = 0 \\
 							& \big \| F x \big \|_{\infty} \ge 1 \\
 				\end{array} \right.,
 		\end{align}
 	where the quadratic term $x\tr Q_d x$ represents the $\Lnorm$-norm of $r[x](d)$ over a given receding horizon, $\D$ is the space of possible disturbance patterns, and the last (non-convex) constraint is concerned with the norm of $G[x]$ as an operator. The matrices $H$ and $F$ are determined by the linear terms of the system dynamics \eqref{ode}, and the positive semidefinite matrix $Q_d$ reflects the nonlinearity signature of the system dynamics in the presence of a disturbance pattern $d$; it depends on $d$ and the nonlinear term $E(\cdot)$ of \eqref{ode}. We refer interested readers to \cite{ref:Mohajerin-FDI-TAC} for details of the derivation of the above program. 
 	
 	For numerical case study, we consider an application of the above FDI design to detect a cyber intrusion in a two-area power network discussed in \cite{ref:Mohajerin_CDC12}. The setup in this example is a simplified version of \cite[Section IV]{ref:Mohajerin_CDC12} where each power area contains one generator. Thus, the state in \eqref{ode} comprises {\small $ X \Let \big[\Delta \phi, \{ \Delta f_i\}_{1:2}, \{\Delta P_{m_i}\}_{1:2}, \{\Delta P_{{agc}_i}\}_{1:2}\big ]\tr$} where $\Delta \phi$ is the voltage angle difference between the ends of the tie line, $\Delta f_i$ the generator frequency, $\Delta P_{m_i}$ the generated mechanical power, and $\Delta P_{{agc}_i}$ the automatic generation control (AGC) signal in each area.\footnote{The symbol $\Delta$ stands for the deviation from the nominal value.}. The system dynamics is modeled in the framework of \eqref{ode}; the details are provided in Appendix \ref{app-B:model}. The disturbance signal $d(\cdot)$ represents a load deviation that may occur in the first area. The signal $f$ models the intrusion signal in the AGC of the first area, and the measurement signals are the frequencies and output power of the turbines, i.e., {\small $Y = \big[\{ \Delta f_i\}_{1:2}, \{\Delta P_{m_i}\}_{1:2} \big]\tr$}. For a given horizon $T>0$, we consider the class of disturbance signatures 
 	\begin{align*}
 		 \D \Let \bigg\{ d:[0,T] \ra \R ~\Big |~ \exists \alpha \in [0,1], ~ d(t) \Let \sum_{k = 0}^{p}  a_k(\alpha)\cos(\frac{2\pi}{T}kt) \bigg\}, 
 	\end{align*}
 	where $a_k(\alpha)$ are the constant coefficients parametrized by $\alpha$. The choice of $\D$ allows one to exploit available spectrum information of the disturbance signals. In this example, motivated by the emphasis on both low and high frequency regions, we assume { $a_k(\alpha) \Let 5\big(\alpha0.5^{k} + (1-\alpha)0.5^{|10 - k |}\big)$}, $p = 30$, and $T = 4 \ \text{sec}$. For scenario generation, we consider a uniform probability distribution for the parameter $\alpha \in [0,1]$, which in fact induces the probability measure $\PP$ on $\D$. Let $d_0 \in \D$ be a disturbance signature with the corresponding parameter $\alpha_0$. It is straightforward to observe that 
 	\begin{align*}
 		\PP \big[\| d - d_0 \|_{\Lnorm} < r \big] & = \PP \bigg [ \frac{T}{2} \sum_{k = 0}^{p} \big| a_k(\alpha) - a_k(\alpha_0) \big|^2  < r^2 \bigg ] \\
 		& =  \PP \Bigg[ |\alpha - \alpha_0| < \frac{\sqrt{2}\ r}{5\sqrt{T\sum_{k = 0}^{p} \big( 0.5^{k} - 0.5^{|10 - k|}\big)^2}}\Bigg]\\
 		& = \PP \big[ |\alpha - \alpha_0| < 0.142 r \big] \ge 0.142 r =: g(r),
 	\end{align*}
 	where the function $g$, denoted in view of Proposition \ref{prop:h}, is an invertible lower bound for the measure of open balls in $\D$. For the particular set of parameters in this example and specific operating region of interest, one can show that the mapping $d \mapsto Q_d$ is Lipschitz continuous with the constant $L_d = 0.02$; see Appendix \ref{app-B:lip} for more details. By virtue of Proposition \ref{prop:h} and normalizing\footnote{Due to the linearity of the filter operator, one can always normalize the filter coefficients with no performance deterioration \cite{ref:Mohajerin-FDI-TAC}.} the optimizer of the $\SCP$ counterpart of the program \eqref{fdi}, we can introduce the ULB candidate 
 	\begin{align*}
 		h(\eps) \Let L_d\ g^{-1}(\eps) =  0.14 \eps.
 	\end{align*}
 	Notice that the Infinite norm constraint in \eqref{fdi} is in fact a non-convex constraint. However, one may view it as the union of a finite number of constraint sets, see \cite[Remark 3.2]{ref:Mohajerin_CDC12}. Therefore, the optimization problem \eqref{fdi} is already in the framework of $\RP$ as introduced in \eqref{rp-cp} where $m$ is the number of rows in matrix $F$. It is remarkable that $m - 1$ equals the degree of the FDI filter chosen a priori. Thanks to the min-max structure of the robust program \eqref{fdi}, the Lipschitz constant of Lemma \ref{lem:lip} for each subprogram of \eqref{fdi} is $\lip = 1$, see Remark \ref{rem:lip}. 
 	
 	In this example, the dimension of the decision variable $x$ is $n = 55$, the number of rows in $F$ is $m = 5$, and the confidence level is set to $\beta = 0.01$. Therefore, to achieve the confidence interval $I(\eps) = h(\eps) = 5 \times 10^{-4}$, we need to set $\eps = 3.57 \times10^{-3} $ which, due to Theorem \ref{thm:feas-nonconvex}, requires to generate $N$ disturbance signatures $d \in \D$ so that
 	\begin{align*}
 		N \ge \min \bigg\{ N\in \N ~\Big|~ \sum_{i=0}^{n-1}  {N \choose i} \eps^{i}(1 - \eps)^{N-i}\leq \frac{\beta}{m} \bigg\} = 22618.
 	\end{align*}
 		 
 	 \begin{figure}[t!]
 	 	\centering
 	 	\subfigure[Scenarios of the disturbance signatures (solid), and intrusion signal (dash)]{\label{fig:df}\includegraphics[scale = 0.22]{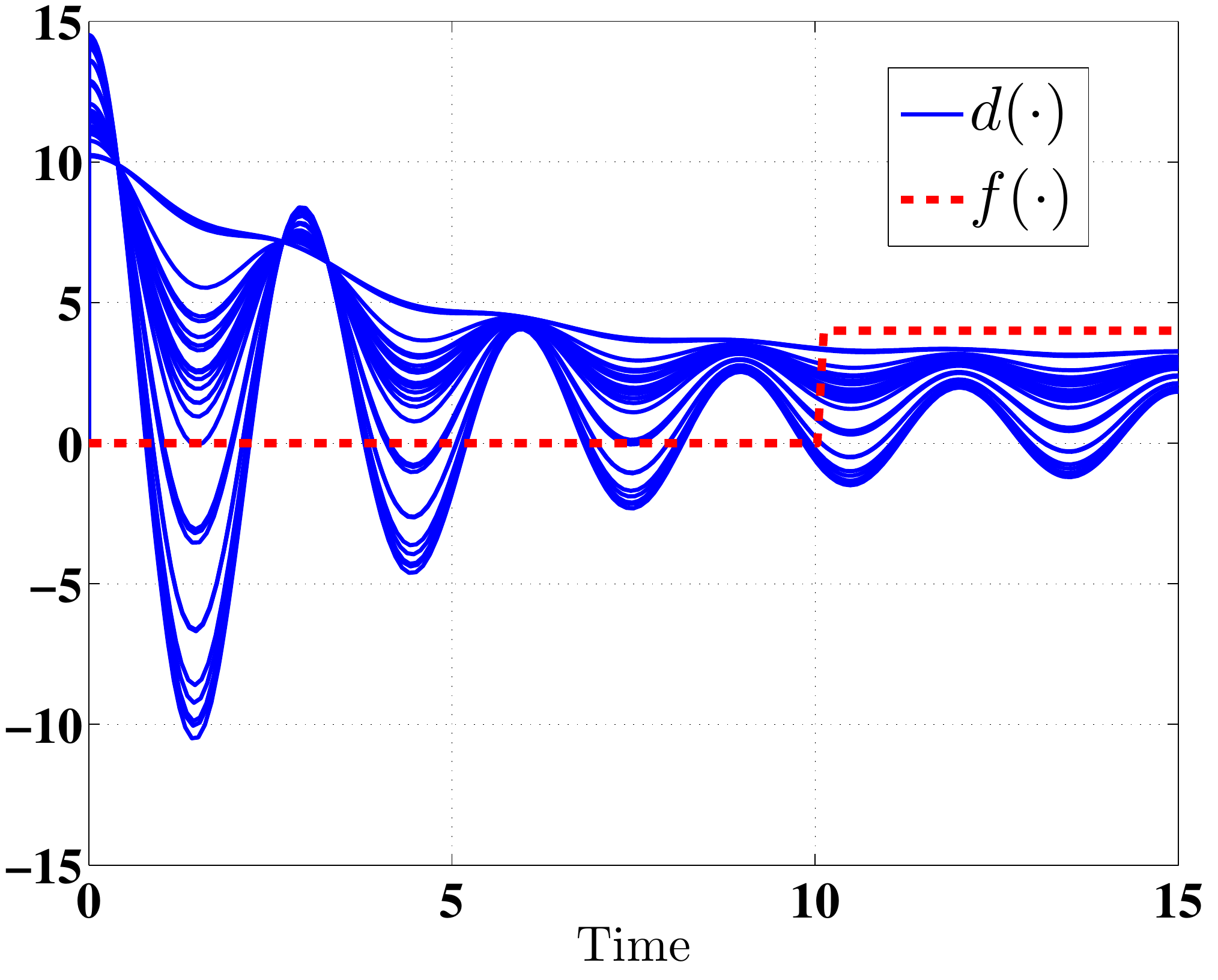}} \quad
 	 	\subfigure[Energy of the filter residual (solid), and the threshold level (dash)]{\label{fig:res}\includegraphics[scale = 0.22]{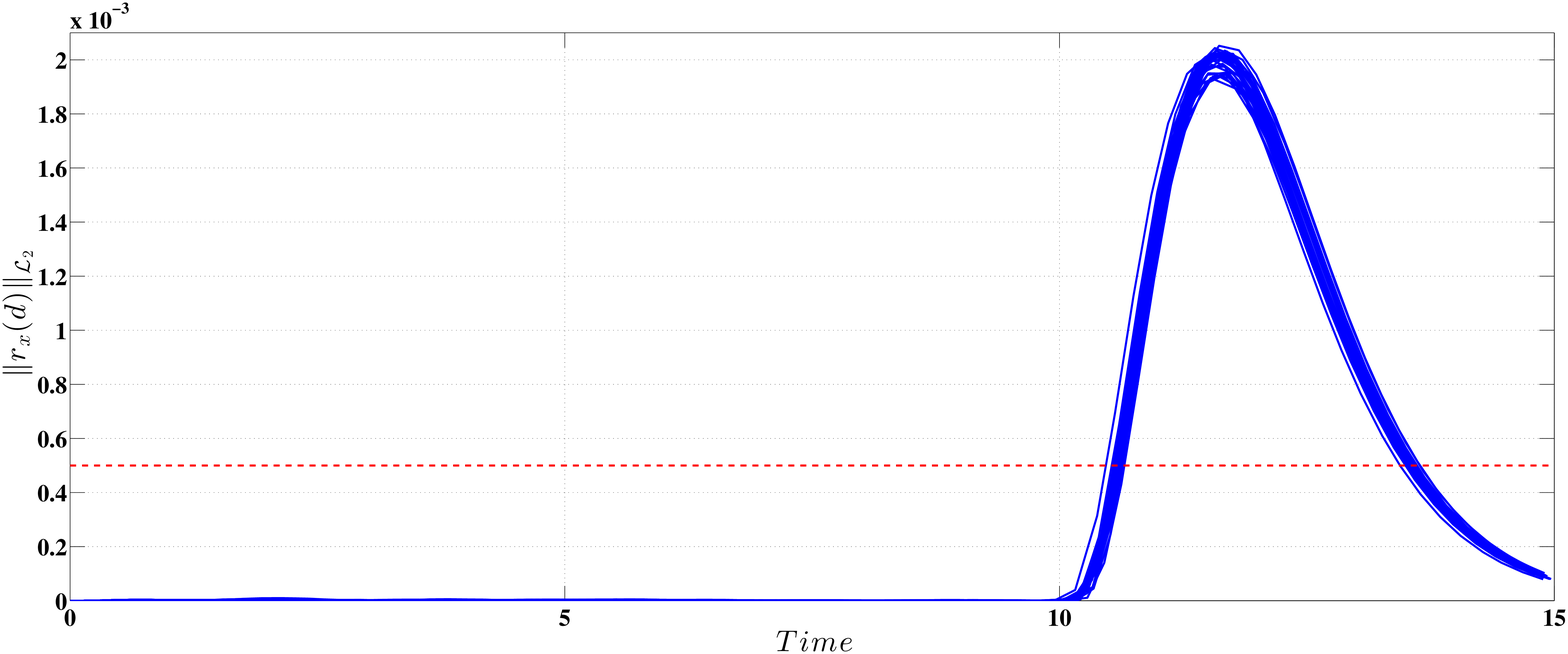}} \quad
 	 	\subfigure[Residual response before the intrusion starts]{\label{fig:res-zoom}\includegraphics[scale = 0.22]{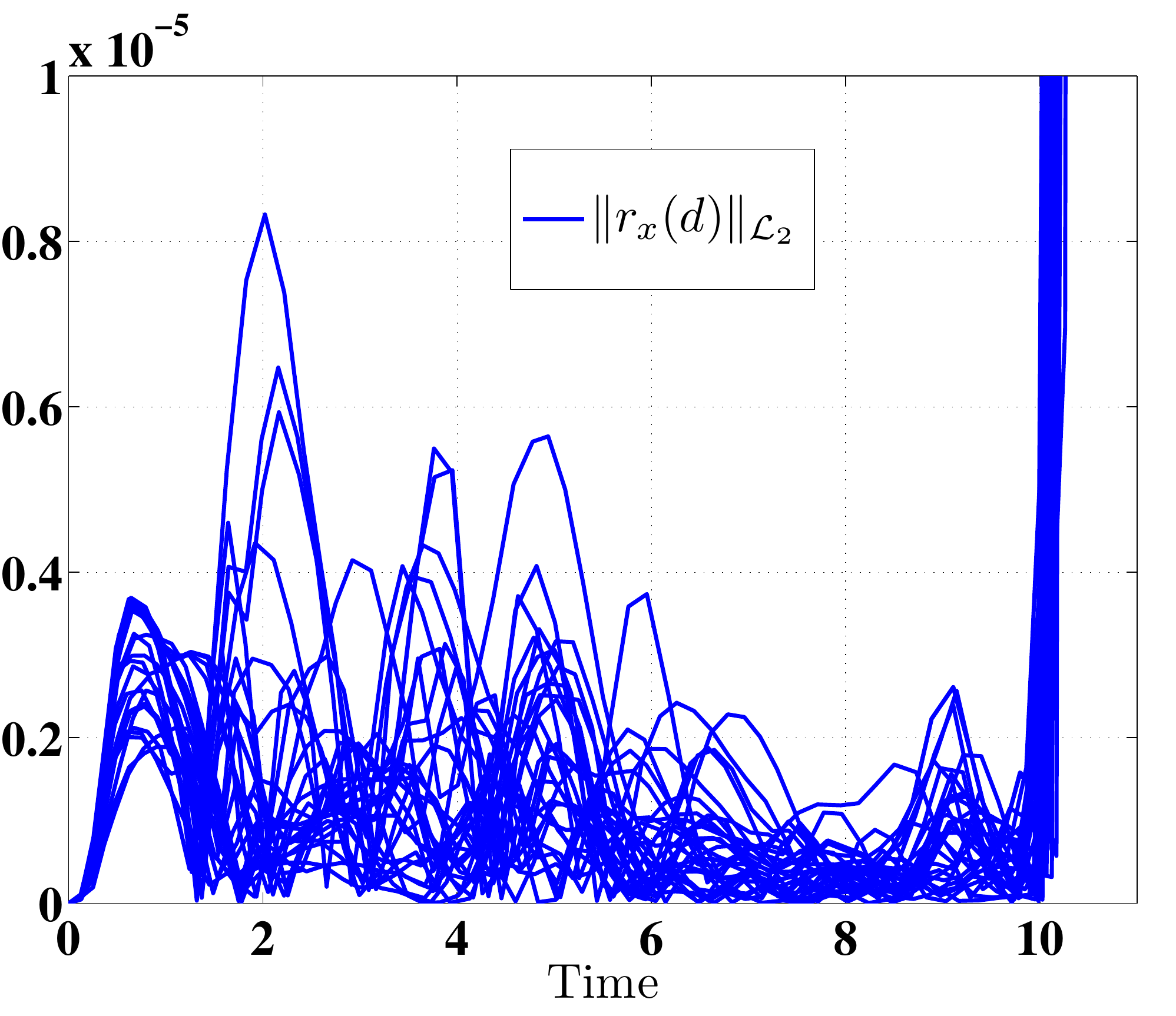}} 
 	 	\caption{Numerical results for Example 2}
 		\label{fig:ex2}
 	 \end{figure}
 	
 	Figures \ref{fig:ex2} demonstrate the numerical results of Example 2 over the course of $15$ seconds. In Figure \ref{fig:df}, 30 different realizations of disturbance inputs as well as an intrusion signal starting from $t = 10$ are shown in solid and dash curves, respectively. Figure \ref{fig:res} depicts the energy of the filter residual for the last $T = 4$ seconds (solid), and the threshold level associated with confidence $\beta = 0.01$ (dash). Notice that the proposed threshold is $\gamma^\star + 0.0005$, where $\gamma^\star$ is the optimal solution of the random counterpart of the program \eqref{fdi} with N = 22618 scenarios. Figure \ref{fig:res-zoom} presents the filter response which is the same figure as \ref{fig:res} but zoomed in on the period prior to the intrusion. 
 	 

\section{Conclusion and Future Direction} \label{sec:conclusion}
 In this article we presented probabilistic performance bounds for both $\RCP$ and $\CCP$ via $\SCP$. The proposed bounds are based on considering the tail probability of the worst-case constraint violation of the $\SCP$ solution as introduced in \cite{ref:KanTak} together with some classical results from perturbation theory of convex optimization. In contrast to earlier approaches, this methodology is, to the best of our knowledge, the first confidence bounds for the objective performance of RCPs and CCPs based on scenario programs. Subsequently, we extended our results to a certain class of non-convex programs allowing for binary decision variables.
 
 For future work, in light of Theorems \ref{thm:perf-rcp} and \ref{thm:perf-ccp}, we aim to study the derivation of ULBs as introduced in Definition \ref{def:tail}. Meaningful ULBs may depend highly on the individual structure of the optimization problems, in particular the uncertainty set and the constraint functions. Another potential direction, as highlighted by Example 1 in Section \ref{ex:1}, is to investigate the relation between the constant $\lip$ in \eqref{lip} and the dual optimizers of the program $\SCP$. 
 
%
 \renewcommand{\thesection}{A}
 \section{Appendix: Technical Proofs} \label{app-A}
 \setcounter{equation}{0}
 \numberwithin{equation}{section}
 
 
 \begin{proof}[Proof of Lemma \ref{lem:feas}]
 	Let $h$ be a ULB as introduced in Definition \ref{def:tail}, $x_0 \sat \CCP$, and $f^*(x_0) \Let \sup_{v \in \D}f(x_0,v)$. By definition of $\CCP$ and $p$, the tail probability of the worst-case violation, we have
 	\begin{align*}
 		p\big(x_0,f^*(x_0)\big) \le \eps ~\Lra~ \inf_{x \in \X} p\big(x,f^*(x_0)\big) \le \eps~\Lra~ f^*(x_0) \le h(\eps) ~\Lra~ x_0 \sat \RCP_{h(\eps)} & \qedhere
 	\end{align*}
 \end{proof}
 
 \begin{proof}[Proof of Proposition \ref{prop:h}]
 	Given $x \in \X$, let $(v_i)_{i \in \N}$ be a sequence in $\D$ so that 
 	$$\limsup_{i\in \N}f(x,v_i) = \sup_{v \in \D} f(x,v).$$
 	Thus, in light of Definition \ref{def:tail} we have
 	\begin{align}
 		\notag p(x,\delta) &= \PP\Big[ \sup_{v \in \D} f(x,v) - f(x,d) < \delta \Big] = \PP\Big[ \limsup_{i\in \N}f(x,v_i) - f(x,d) < \delta \Big] \\
 		& \label{e:1}\ge \PP\Big[ \limsup_{i\in \N} L_d \| v_i - d \| < \delta \Big] \ge \limsup_{i \in \N} \PP \bigg[ \| v_i - d \| < \frac{\delta}{L_d} \bigg] \\
 		\notag & = \limsup_{i \in \N} \PP \Big[ \ball{v_i}{\frac{\delta}{L_d}} \Big] \ge g\big( \frac{\delta}{L_d} \big),
 	\end{align}
 	where the first inequality in \eqref{e:1} follows from the Lipschitz continuity of $f$ with respect to $d$, and the second inequality in \eqref{e:1} is due to Fatou's lemma~\cite[p.~23]{ref:Rudin-87}. Hence, in view of the ULB definition and  the above analysis, we arrive at
 	\begin{align*}
 		\sup \Big \{ \delta \in \R_+ ~\big|~ \inf_{x \in \X} p(x,\delta) \le \eps \Big \} \le \sup \Big \{ \delta \in \R_+ ~\big|~ g\big( \frac{\delta}{L_d} \big) \le \eps \Big \} = L_d g^{-1}(\eps). &\qedhere
 	\end{align*}
 \end{proof}

 \begin{proof}[Proof of Lemma \ref{lem:lip}]
	It is well-known that under the strong duality condition the mapping $\gamma \mapsto \RCP_{\gamma}$, the so-called perturbation function, is Lipschitz continuous with the constant $\|\lambda^\star\|_1$ where $\lambda^\star$ is a dual optimizer of the program $\RCP$; see \cite[p.\ 250]{ref:Boyd-04} for the proof and \cite[Section~28]{ref:Rockafellar-97} for more details in this direction. Now Lemma \ref{lem:lip} follows from \cite[Lemma 1]{ref:Nedic-08}, which essentially implies $\|\lambda^\star\|_1 \le \lip$ where $\lip$ is the constant \eqref{lip} corresponding to any Slater point in the sense of Assumption \ref{a:duality}. 
 \end{proof}
 	
 	To prove Proposition \ref{prop:meas}, we need some preliminaries. 	
 	
 	
 	\begin{Lem}
 	 \label{lem:lsc}
 	 	Let $\C$ be the set of all lower semicontinuous functions from $\X \subset \R^n$ to $\R$. Consider the mapping $J : \Lcon \ra \R$ defined by the optimization program
 	 	\begin{align}
 	 	 \label{J}
 	 		\left\{ \begin{array}{lll}
 	 		  J(g) \Let &\min\limits_{x}					& c\tr x \\
 	 			&\emph{s.t. } 			& g(x) \le 0\\
 	 			&						& x\in \X 
 	 				\end{array}
 	 		\right..
 		\end{align}
 		Then, the function $J$ is measurable where the space of $\C$ is endowed with the infinite norm and the respective Borel $\sigma$-algebra.\footnote{Under assumptions of Section \ref{sec:problem}, one can show a stronger assertion that the mapping $g \mapsto J(g)$ is indeed lower semicontinuous; see for instance \cite[Theorem 4.3.2, p.\ 67]{ref:Klatte-82}. Thanks to a personal communication with Diethard Klatte, it turns out that the statement can be even extended to continuity if Assumption \ref{a:duality} also holds.} 
 	\end{Lem}
 
	\begin{proof}
		The proof is an application of \cite[Theorem 14.37, p.\ 664]{ref:Rockafellar-10}. Let us define the set-valued mapping $S:\C \rra \X \times \R$ as follows:	 	
		 	 	\begin{align*}
		 	 		S(g)\Let \big\{(x,\alpha) \in \X \times \R ~\big|~ \{g(x) \le 0\} ~ \& ~ \{c\tr x \le \alpha\} \big\}. 
		 	 	\end{align*}
 	 	We first show that $S$ is a normal integrand in the sense of \cite[Definition 14.27, p.\ 661]{ref:Rockafellar-10}. Since $g$ is lower semicontinuous, then $S$ is clearly closed-valued. We then only need to show that $S$ is measurable according to \cite[Definition 14.1, p.\ 643]{ref:Rockafellar-10}. Let $O \subset \X \times \R$ be an open set, $(x_0,\alpha_0) \in O$ and $g_0 \in S^{-1}(x_0,\alpha_0)$. Observe that for sufficiently small $\eps>0$ we have $\ball{g_0}{\eps} \subset S^{-1}(O)$ where $\ball{g_0}{\eps}\Let \{g \in \C ~|~ \sup_{x \in \X} \|g(x) - g_0(x) \| \le \eps\}$, that implies that $S^{-1}(O)$ is open and in particular measurable. Thereby, $S$ is measurable and hence a normal integral. Now the desired measurability readily follows from \cite[Theorem 14.37, p.\ 664]{ref:Rockafellar-10}.
 	 \end{proof}
 
%
 	\begin{Lem}
 	 \label{lem:meas}
 	 	Let $\phi : \R^n \ra \R$ be a strictly convex function, and $\wt J : \Lcon \ra \R$ defined as follows:
 		 	\begin{align}
 	 		\label{J-tilde}
 		 	 	\left\{ \begin{array}{lll}
 	 			  \wt J(g) \Let &\min\limits_{x}		& \phi(x) \\
 	 							&\emph{s.t. } 			& g(x) \le 0\\
 	 							&			& c\tr x \le J(g)\\ 
 	 							&						& x\in \X 
 	 					\end{array} \right.	,		
 	 		\end{align}
 	 	where $J(g)$ is the function introduced in \eqref{J}. Let $\wt x^\star(g)$ denote the set of optimizers of the program \eqref{J-tilde}. Then, the mapping $\Lcon \ni g \mapsto \wt x^\star \in \R^d$ is a measurable singleton.
 	\end{Lem}
 	
 	
 	\begin{proof}
		Let us define the set-valued mapping $S:\Lcon \rra \X \times \R$ 
 	 	\begin{align*}
 	 		S(g)\Let \big\{(x,\alpha) \in \X \times \R ~\big|~ \{g(x) \le 0\} ~ \& ~ \{c\tr x -J(g) \le 0\} ~\&~ \{\phi(x) \le \alpha\}  \big\}. 
 	 	\end{align*}
 	 	By virtue of the measurability of the mapping $g \mapsto J(g)$ in Lemma \ref{lem:lsc} and along the same line of its proof, we know that $S$ is a normal integral. Now, by \cite[Theorem 14.37, p.\ 664]{ref:Rockafellar-10} the existence of a measurable selection for the optimizer $\wt{x}^\star(g)$ as a function of $g \in \C$ is guaranteed. On the other hand, since $\phi: \R^n \ra \R$ is strictly convex, the minimizer of the program \eqref{J-tilde} is unique. Therefore, $\wt{x}^\star(g)$ is a singleton and the desired measurability property follows at once. 
 	\end{proof}
 	
 	We now have all the required results to prove Proposition \ref{prop:meas}:	
 	
 	\begin{proof}[Proof of Proposition \ref{prop:meas}]
 		Let $g : \D^N \ra \Lcon$ defined as 
 		 \begin{align}
 		 \label{g}
 		 	g(d_1,\cdots,d_N ) \Let \max_{i \in \{1,\cdots, N\}} f(x,d_i).
 		 \end{align}
 		The measurability of the mapping \eqref{g} is ensured by the measurability assumption of the mapping $d \mapsto f(x,d)$ for each $x$. It is straightforward to observe that the optimizer of the program \eqref{scp-2} can be viewed as the composition $\wt{x}^\star_N = \wt x^\star\circ g(d_1,\cdots,d_N )$ where $\wt x^\star$ is the optimizer of the program \eqref{J-tilde} and $g$ is defined as in \eqref{g}. Hence, the desired implication follows directly from the measurability of the mapping \eqref{g} and Lemma \ref{lem:meas}. 
 	\end{proof}
 
 \renewcommand{\thesection}{B}
 \section{Appendix: Details of Example 2}\label{app-B}
 
 	
 	\subsection{Mathematical model description}
 	\label{app-B:model}
 	The two-area power network is described by the set of nonlinear ordinary differential equations
 	{ \begin{align*}
 		\Delta \dot{\phi} &= 2\pi (\Delta f_1 - \Delta f_2), \\
 		\Delta \dot{f}_i & = \frac{f_0}{2H_iS_{B_i}} \Big( - \frac{1}{D_{i}}\Delta f_i - P_{T}\sin\Delta\phi +  \Delta P_{m_i} - \Delta P_{load_i} \Big), \\
 		\Delta \dot{P}_{m_i} &= \frac{1}{T_{ch_i}}\Big(  - \frac{1}{S_i} \Delta f_i - \Delta P_{m_i} + \Delta P_{agc_i} \Big),\\
 		\Delta \dot{P}_{agc_i} &= \Big(\frac{1}{D_{i}} \frac{C_{i}f_0}{2S_i H_i S_{B_i}} - \frac{1}{S_i}\frac{1}{T_{N_i}} \Big) \Delta f_i  \\
 		& \qquad - \frac{C_{i}f_0}{2S_i H_i S_{B_i}} \big( \Delta P_{m_i} - \Delta P_{load_i} \big) - \frac{C_{i}f_0}{2S_i H_i S_{B_i}} \Delta P_{agc_i}  \\
 		&\qquad - \Big( \frac{1}{T_{N_i}} -\frac{C_{i}f_0}{2S_i H_i S_{B_i}} \Big) P_{T}\sin\Delta\phi - 2\pi C_{i} P_{T} (\Delta f_1 - \Delta f_2) \cos\Delta\phi,  
 	\end{align*} }
 	where $i \in \{1,2\}$ is the index of each area, {\small $ X \Let \big[\Delta \phi, \{ \Delta f_i\}_{1:2}, \{\Delta P_{m_i}\}_{1:2}, \{\Delta P_{{agc}_i}\}_{1:2}\big ]\tr \in \R^7$} is the state vector, and the constant parameters in this example are chosen the same for both areas as {\small $T_{ch_i} = 5 \ \text{sec}, S_{B_i} = 1.8 \ \text{GW}$, $f_0 = 50 \ \text{Hz}$, $H_i = 6.5 \ \text{sec}$, $D_i = 428.6 \ \text{Hz/GW}$, $S_i = 1.389 \ \text{Hz/GW}$, $C_i = 0.1$, $T_{N_i} = 30$, $P_T = 0.15 \ \text{GW}$}. We refer to \cite{ref:Mohajerin_ACC10} for physical interpretation of these parameters and more details on the model equations. In the example, we assume that $\Delta P_{load_1} = d$ where $d \in \D$ is the disturbance signal and $\Delta P_{load_2} \equiv 0$.
 	
 	
 	\subsection{Lipschitz constant of the mapping $d \mapsto Q_d$} 
 	\label{app-B:lip}
 	This mapping can be viewed in two steps: $d \mapsto E(X)$ and $E(X) \mapsto Q_d$ where $X$ is the solution process in the presence of the disturbance input $d$, and $E$ the nonlinear term of the ODE \eqref{ode}. The key step is to approximate the Lipschitz constant of the first mapping $d \mapsto E(X)$. The classical result of the continuity of the ODEs solution, obtained by Lipschitz continuity of the vector field and Gronwall's inequality, turns out to be too conservative in this case. We then invoke a Lyapunov-like approach to address this issue more efficiently. Let us define the shorthand $h(X,d) \Let E(X) + AX + B_d d$. Suppose there exist a function $V: \R^7 \times \R^7 \ra \R_+$ and positive constants $\kappa, \rho$ so that for every $X, \wt X \in \R^7$ and $d, \wt d \in \R$
 	{\begin{subequations}
 	\label{V}
 	\begin{align}
 		\label{V-a1}\big \| E(X) - E(\wt X) \big\|^2 & \le V(X, \wt X) \\
 		\label{V-a2}\partial_X V(X, \wt X) h(X, d) + \partial_{\wt X} V(X, \wt X) h(\wt X ,\wt d) &\le -\kappa V(X, \wt X) + \rho \big| d - \wt d \big|.
 	\end{align}
 	\end{subequations}}
 	Using standard Gronwall's inequality, one can show that under conditions \eqref{V} we have 
 	{\begin{align*}
 		 \big \| E(X) - E(\wt X) \big\|_{\Lnorm}^2 & \le \int\limits_{0}^{T} \big \| E\big(X(t)\big) - E\big(\wt X(t)\big) \big\|^2 \diff t \le \int\limits_{0}^{T} V\big(X(t), \wt X(t)\big) \diff t \\
 		& \le  \rho \int\limits_{0}^{T} \e^{-\kappa t} \int\limits_{0}^{t} \big| d(s) - \wt d(s)\big| \diff s  \diff t \le \frac{2 \rho}{3}T\sqrt{T} \| d - \wt d\|_{\Lnorm}. 
 	\end{align*} }
 	In \cite[Theorem 3.3]{ref:ZamMoh-TAC-13}, a similar technique is discussed in more detail to establish a connection between the Lyapunov function and continuity of the solution trajectories. In order to find a Lyapunov function in the above sense, we limit our search domain to the quadratic functions, i.e., {\small $V(X, \wt X) = (X - \wt X)\tr Q (X - \wt X)$} for some positive semidefinite matrix $Q$. It is not difficult to deduce that the nonlinear term $E$ effectively depends only on the state $\Delta \phi$. Hence, to fulfill the requirement \eqref{V-a1} it suffices to guarantee $Q \succeq v v\tr$ where $v = [0, 0, 1, 0, 0, 0, 0]\tr$. Setting $\kappa = 0.01$, we then solve the set of linear matrix inequalities (LMIs)
 	\begin{align*}
 	\left\{ \begin{array}{lll} \vspace{0.2mm}
 		 \min\limits_{\sigma, Q}	& \sigma \\
 		 			\text{s.t. } & Q A\tr + A Q \preceq -\kappa Q\\
 		 			& v v\tr \preceq Q \preceq \sigma I\\
 		\end{array} \right.,
 	\end{align*}
 	 which provides a local Lyapunov function in the sense of \eqref{V}. Note that one can always extract the linear part of $E$ and add it to the matrix A. Now, by numerical inspection, it turns out that for the specific system parameters of this example, $V$ obtained from the above LMIs is a Lyapunov function in the domain of {\small $\Delta f_i \in [-0.1, 0.1]\ \text{Hz}$, $\Delta \phi \in [-10^\circ, 10^\circ]$, $\Delta p_{m_i} \in [-10, 10]\ \text{MW}$, $\Delta p_{agc_i} \in [-15, 15]\ \text{MW}$}. Therefore, the parameter $\rho$ in \eqref{V-a2} can be numerically approximated via the optimal $\sigma$ in the LMIs together with matrix $B_d$ and the region of interest described above. Besides, since the FDI filter is a stable linear time invariant transfer function with normalized coefficients, the Lipschitz constant of the second mapping $E(X) \mapsto Q_d$ can be explicitly computed based on the filter denominator which is fixed prior to the design procedure; see \cite[Lemma 4.5]{ref:Mohajerin-FDI-TAC}. 

	\bibliographystyle{alpha}

	\bibliography{ref_FDI,ref_mine,ref_SOC}

\newcommand{\etalchar}[1]{$^{#1}$}
\begin{thebibliography}{MEVM{\etalchar{+}}10}

\bibitem[ATC09]{ref:Alamo-09}
Teodoro Alamo, Roberto Tempo, and Eduardo~F. Camacho.
\newblock Randomized strategies for probabilistic solutions of uncertain
  feasibility and optimization problems.
\newblock {\em IEEE Trans. Automat. Control}, 54(11):2545--2559, 2009.

\bibitem[BGFB94]{ref:Boyd-94}
Stephen Boyd, Laurent~El Ghaoul, Eric Feron, and Venkataramanan Balakrishnan.
\newblock {\em Linear Matrix Inequalities in System and Control Theory}.
\newblock Society for Industrial and Applied Mathematics: SIAM studies in
  applied mathematics. Society for Industrial and Applied Mathematics, 1994.

\bibitem[BGK{\etalchar{+}}83]{ref:Klatte-82}
Bernd Bank, J\"urgen Guddat, Diethard Klatte, Bernd Kummer, and Klaus Tammer.
\newblock {\em Nonlinear parametric optimization}.
\newblock Birkh\"auser Verlag, Basel, 1983.

\bibitem[Bil95]{ref:Bil-95}
Patrick Billingsley.
\newblock {\em Probability and measure}.
\newblock Wiley Series in Probability and Mathematical Statistics. John Wiley
  \& Sons Inc., New York, third edition, 1995.
\newblock A Wiley-Interscience Publication.

\bibitem[BS06]{ref:Bertsimas-06}
Dimitris Bertsimas and Melvyn Sim.
\newblock Tractable approximations to robust conic optimization problems.
\newblock {\em Mathematical Programming}, 107(1-2):5--36, 2006.

\bibitem[BtN98]{ref:BenTal-98}
Aharon Ben-tal and Arkadi Nemirovski.
\newblock Robust convex optimization.
\newblock {\em Mathematics of Operations Research}, 23:769--805, 1998.

\bibitem[BtN99]{ref:BenTal-99}
Aharon Ben-tal and Arkadi Nemirovski.
\newblock Robust solutions of uncertain linear programs.
\newblock {\em Operations Research Letters}, 25:1--13, 1999.

\bibitem[BtNR01]{ref:BenTal-01}
Aharon Ben-tal, Arkadi Nemirovski, and C.~Roos.
\newblock Robust solutions of uncertain quadratic and conic-quadratic problems.
\newblock In {\em Solutions of Uncertain Linear Programs: Math. Program}, pages
  351--376. Kluwer, 2001.

\bibitem[BV04]{ref:Boyd-04}
Stephen Boyd and Lieven Vandenberghe.
\newblock {\em Convex Optimization}.
\newblock Cambridge University Press, New York, NY, USA, 2004.

\bibitem[Cal09]{ref:Cal-09}
Giuseppe~C. Calafiore.
\newblock A note on the expected probability of constraint violation in sampled
  convex programs.
\newblock In {\em 18th IEEE International Conference on Control Applications
  Part of 2009 IEEE Multi-conference on Systems and Control}, pages 1788 --
  1791, july 2009.

\bibitem[Cal10]{ref:Cal-10}
Giuseppe~Carlo Calafiore.
\newblock Random convex programs.
\newblock {\em SIAM J. Optim.}, 20(6):3427--3464, 2010.

\bibitem[CC05]{ref:CalCam-05}
Giuseppe Calafiore and M.~C. Campi.
\newblock Uncertain convex programs: randomized solutions and confidence
  levels.
\newblock {\em Math. Program.}, 102(1, Ser. A):25--46, 2005.

\bibitem[CC06]{ref:CalCam-06}
Giuseppe~C. Calafiore and Marco~C. Campi.
\newblock The scenario approach to robust control design.
\newblock {\em IEEE Trans. Automat. Control}, 51(5):742--753, 2006.

\bibitem[CG08]{ref:CamGar-08}
M.~C. Campi and S.~Garatti.
\newblock The exact feasibility of randomized solutions of uncertain convex
  programs.
\newblock {\em SIAM J. Optim.}, 19(3):1211--1230, 2008.

\bibitem[CG11]{ref:Campi-11}
M.~C. Campi and S.~Garatti.
\newblock A sampling-and-discarding approach to chance-constrained
  optimization: feasibility and optimality.
\newblock {\em J. Optim. Theory Appl.}, 148(2):257--280, 2011.

\bibitem[CGP09]{ref:Campi-09}
Marco~C. Campi, Simone Garatti, and Maria Prandini.
\newblock The scenario approach for systems and control design.
\newblock {\em Annual Reviews in Control}, (2):149--157, 2009.

\bibitem[CLF12]{ref:Calafiore-12}
Giuseppe~C. Calafiore, D.~Lyons, and L.~Fagiano.
\newblock On mixed-integer random convex programs.
\newblock In {\em Decision and Control (CDC), 2012 IEEE 51st Annual Conference
  on}, pages 3508--3513, 2012.

\bibitem[DPR13]{ref:Duf-13}
Fran{\c{c}}ois Dufour and Tom{\'a}s Prieto-Rumeau.
\newblock Finite linear programming approximations of constrained discounted
  {M}arkov decision processes.
\newblock {\em SIAM J. Control Optim.}, 51(2):1298--1324, 2013.

\bibitem[Gah96]{ref:Gahinet-96}
Pascal Gahinet.
\newblock Explicit controller formulas for lmi-based h-infinity synthesis.
\newblock {\em Automatica}, 32(7):1007--1014, July 1996.

\bibitem[GOL98]{ref:Ghaoui-98}
Laurent~El Ghaoui, Francois Oustry, and HervŽ Lebret.
\newblock Robust solutions to uncertain semidefinite programs.
\newblock {\em SIAM Journal on Optimization}, 9(1):33--52, 1998.

\bibitem[KT12]{ref:KanTak}
Takafumi Kanamori and Akiko Takeda.
\newblock Worst-case violation of sampled convex programs for optimization with
  uncertainty.
\newblock {\em Journal of Optimization Theory and Applications},
  152(1):171--197, 2012.

\bibitem[LVLM08]{ref:LecLyg-08}
Andrea Lecchini-Visintini, John Lygeros, and Jan~M. Maciejowski.
\newblock {Approximate domain optimization for deterministic and expected value
  criteria}.
\newblock Technical report, March 2008.
\newblock [Online]. Available:
  \url{http://control.ee.ethz.ch/index.cgi?page=publications;action=details;id=3048}.

\bibitem[LVLM10]{ref:LecLyg-10}
Andrea Lecchini-Visintini, John Lygeros, and Jan~M. Maciejowski.
\newblock Stochastic optimization on continuous domains with finite-time
  guarantees by {M}arkov chain {M}onte {C}arlo methods.
\newblock {\em IEEE Trans. Automat. Control}, 55(12):2858--2863, 2010.

\bibitem[MA04]{ref:Lex-04}
R~Timothy Marler and Jasbir~S Arora.
\newblock Survey of multi-objective optimization methods for engineering.
\newblock {\em Structural and multidisciplinary optimization}, 26(6):369--395,
  2004.

\bibitem[MEL13]{ref:Mohajerin-FDI-TAC}
Peyman Mohajerin~Esfahani and John Lygeros.
\newblock A tractable fault detection and isolation approach for nonlinear
  systems with probabilistic performance.
\newblock Technical report, February 2013.
\newblock [Online]. Available:
  \url{http://control.ee.ethz.ch/index.cgi?page=publications&action=details&id=4344}.

\bibitem[MEVAL12]{ref:Mohajerin_CDC12}
Peyman Mohajerin~Esfahani, Maria Vrakopoulou, Goran Andersson, and John
  Lygeros.
\newblock A tractable nonlinear fault detection and isolation technique with
  application to the cyber-physical security of power systems.
\newblock In {\em 51th IEEE Conference Decision and Control}, 2012.
\newblock [Online]. Full version:
  \url{http://control.ee.ethz.ch/index.cgi?page=publications;action=details;id=4196}.

\bibitem[MEVM{\etalchar{+}}10]{ref:Mohajerin_ACC10}
Peyman Mohajerin~Esfahani, Maria Vrakopoulou, Kostas Margellos, John Lygeros,
  and Goran Andersson.
\newblock Cyber attack in a two-area power system: Impact identification using
  reachability.
\newblock In {\em American Control Conference}, pages 962 -- 967, 2010.

\bibitem[NO08]{ref:Nedic-08}
Angelia Nedi{\'c} and Asuman Ozdaglar.
\newblock Approximate primal solutions and rate analysis for dual subgradient
  methods.
\newblock {\em SIAM J. Optim.}, 19(4):1757--1780, 2008.

\bibitem[PRC12]{ref:PagCam-12}
B.~K. Pagnoncelli, D.~Reich, and M.~C. Campi.
\newblock Risk-return trade-off with the scenario approach in practice: a case
  study in portfolio selection.
\newblock {\em J. Optim. Theory Appl.}, 155(2):707--722, 2012.

\bibitem[Pr{\'e}95]{ref:Prekopa-95}
Andr{\'a}s Pr{\'e}kopa.
\newblock {\em Stochastic Programming}.
\newblock Mathematics and Its Applications. Springer, 1995.

\bibitem[Roc97]{ref:Rockafellar-97}
R.~Tyrrell Rockafellar.
\newblock {\em Convex analysis}.
\newblock Princeton Landmarks in Mathematics and Physics Series. PRINCETON
  University Press, 1997.

\bibitem[Rud87]{ref:Rudin-87}
Walter Rudin.
\newblock {\em Real and complex analysis}.
\newblock Mathematics series. McGraw-Hill, 1987.

\bibitem[RW10]{ref:Rockafellar-10}
R.~Tyrrell Rockafellar and Roger~J.B. Wets.
\newblock {\em Variational Analysis}.
\newblock Grundlehren der mathematischen Wissenschaften. Springer, 2010.

\bibitem[SDR09]{ref:Shapiro-09}
Alexander Shapiro, Darinka Dentcheva, and Andrzej Ruszczy{\'n}ski.
\newblock {\em Lectures on Stochastic Programming: Modeling and Theory}.
\newblock MPS-SIAM series on optimization. Society for Industrial and Applied
  Mathematics, 2009.

\bibitem[Sha03]{ref:Shao-03}
Jun Shao.
\newblock {\em Mathematical statistics}.
\newblock Springer Texts in Statistics. Springer-Verlag, New York, second
  edition, 2003.

\bibitem[ZMEAL13]{ref:ZamMoh-TAC-13}
Majid Zamani, Peyman Mohajerin~Esfahani, Alessandro Abate, and John Lygeros.
\newblock Symbolic models for stochastic control systems.
\newblock Technical report, February 2013.
\newblock [Online]. Available: \url{http://arxiv.org/abs/1302.3868}.

\end{thebibliography}

\end{document}